\documentclass[11pt,twoside]{article}
\usepackage{latexsym}
\usepackage{amssymb,amsbsy,amsmath,amsfonts,amssymb,amscd}
\usepackage{graphicx,color}
\usepackage{float,url}
\usepackage{cancel}
\usepackage[colorlinks=true, citecolor =violet, linkcolor = violet, urlcolor  = violet]{hyperref}
\usepackage{xcolor}
\usepackage{tikz}
\usepackage{bbm}
\usepackage{amsthm}

\newtheorem{theorem}{Theorem}[section]
\newtheorem{corollary}[theorem]{Corollary}
\newtheorem{lemma}[theorem]{Lemma}

\newtheorem{prop}[theorem]{Proposition}

\newtheorem{definition}[theorem]{Definition}
\newtheorem{remark}[theorem]{Remark}

\newcommand{\Cc}{\mathcal{C}}
\newcommand{\Ee}{\mathcal{E}}

\newcommand{\N}{\mathbb{N}}

\newcommand{\R}{\mathbb{R}}

\newcommand{\eps}{\varepsilon}
\newcommand{\fa}{\forall}

\newcommand{\reg}{\omega_\varepsilon}

\newcommand{\ureg}{u_\varepsilon}

\newcommand{\signe}{\text{sgn}}
\newcommand{\testfunc}{\mathcal{C}_c^\infty}

\newcommand{\commentout}[1]{}

\definecolor{bluegreen}{rgb}{0.0, 0.6, 0.6}
\title{A Fokker-Planck equation with superlinear drift at infinity \\ for integrate-and-fire model}

\date{\today}

\author{Beno\^\i t Perthame%
\thanks{Sorbonne Universit\'e, CNRS, Universit\'e de Paris Cit\'e, Inria,
Laboratoire Jacques-Louis Lions (LJLL), F-75005 Paris, France.
\newline\texttt{benoit.perthame@sorbonne-universite.fr}}
\and
Cl\'ement Rieutord%
\:\thanks{Corresponding author, Sorbonne Universit\'e, CNRS, Universit\'e de Paris Cit\'e,
Laboratoire Jacques-Louis Lions (LJLL), F-75005 Paris, France.
\newline\texttt{clement.rieutord@sorbonne-universite.fr}}
\and
Delphine Salort%
\thanks{Sorbonne Universit\'e, CNRS, Universit\'e de Paris Cit\'e,
Laboratoire Jacques-Louis Lions (LJLL), F-75005 Paris, France.
\newline\texttt{dsalort@gmail.com}}
}

\begin{document}

	\maketitle

	\begin{abstract}
	    The Integrate-and-Fire model is a Fokker-Planck equation arising in neuroscience. It describes the evolution of the probability density of the neuronal membrane potential and fitting has shown that the inclusion of a {\em superlinear drift} provides the most realistic description. To make sense of this, we propose to set the equation on the full line, the neural activity being described by the flux at infinity. This framework serves as a model extension of the classical Noisy Integrate-and-Fire model, with a fixed firing potential.
We first establish the well-posedness of the solution, establish the boundary condition at infinity which is the major difficulty. Then,  state rigorously the  entropy dissipation property. Finally, using Doeblin’s method, we prove the exponential convergence of the solution toward the unique stationary state in full generality.
	\end{abstract}
	
\noindent{\makebox[1in]\hrulefill}\newline


 	\tableofcontents
	
	\section{Introduction}
 
 In recent literature, many works have focused on the {\em integrate-and-fire} (I\&F) models for neurons potentials, which now constitute a widely studied reference framework (see the survey \cite{CR2025}). However, several approaches from biophysics and computational neuroscience suggest a specificity, namely that intrinsic  dynamics with super-linear growth for large potentials, notably quadratic or exponential, would provide a more realistic description of certain neuronal behaviors. This motivation is notably inspired by the work of \cite{BL_2003, BretteG2005} on the Quadratic  or Exponential Integrate-and-Fire  model, which showed that introducing a super-linearity in the drift better captures neuronal firing dynamics and certain oscillations, compared to classical linear models.  We propose a partial differential equation (PDE) that incorporates this super-linear component in order to develop a theory and a qualitative analysis of neuronal dynamics within this framework. To date, this theoretical framework has never been subject to a rigorous mathematical study via a Fokker-Planck type PDE. The associated linear model we propose is formally as follows:
	\begin{equation}\label{eq:FPSL lin}
		\left\{\begin{aligned}
			& \frac{\partial u}{\partial t}(t,x) +\frac{\partial}{\partial x}\big((h(x)u(t,x)\big) - \frac{\partial ^2u}{\partial x^2}(t,x) = \delta_{0}(x)N(t), \quad x\in \mathbb{R}, t\geq 0,
			\\
			& N(t) = \lim_{x\rightarrow \infty}h(x)u(t,x),
			\\
			& u(t=0,x) = u_0(x)\in L^1(\R).
		\end{aligned}\right.
	\end{equation}
The specificity is to work on the full line with a flux of firing neurons $h(x)u(x)$ at infinity which occurs because we assume that $\int_{\cdot}^{\infty} \frac{1}{h(x)}dx <+\infty$. This specificity  also makes the mathematical interest and new difficulty for the analysis of the model. 

\subsection{Motivations}
The study of this equation fundamentally differs, in its structure, from the classical linear-drift I\&F model introduced in \cite{BrHa}, 
	\begin{equation*}
		\left\{\begin{aligned}
			& \frac{\partial p}{\partial t}(t,x) +\frac{\partial}{\partial x}(h(x) p(t,x)) - \frac{\partial ^2p}{\partial x}(t,x) = \delta_{V_R}(x)N(t), \qquad t\geq 0, \; x\leq V_F,\\
			& p(t,V_F) = 0\\
			& N(t) = -\frac{\partial p}{\partial x}(t,V_F).
		\end{aligned}\right.
	\end{equation*}
This equation describes the probability density of neurons, according to their membrane potential~$x$. 	When a neuron reaches the potential of firing, $V_F$, then it is instantaneously re-injected with a membrane potential of reset, $V_R< V_F$ 
(we have chosen $V_R=0$ in~\eqref{eq:FPSL lin} for simplicity). 

In such a model, the mathematical theory does not take into account that the drift is quadratic or exponential. The super-linearity of the drift is a concept for $x\to \infty$. This leads us to consider the equation on the full line and compute the flux of firing neurons $N(t)$ through a source term located at infinity, unlike in the classical case where it is imposed by a Dirichlet boundary condition at $V_F$. 

Notice that the novelty is the {\em flux condition at infinity}, which justifies the superlinear growth of $h(x)$ and impose to extend the I\&F equation to the full line with $V_F=\infty$. This extension to the full line has also been used,  mainly to understand when the drift depends on the neural activity $N(t)$, the complex possible dynamics including periodic solutions, \cite{CaPe, DPSZ2024,JGLiuZZ21}.

We aim to develop a theory around this new formulation \eqref{eq:FPSL lin} with a super-linear drift, in a first step, in the linear setting, that is, in term of modeling, the neurons are independent the one with each others. 
\\

\subsection{Assumptions}

	In order to catch the superlinear behaviour of the drift, we consider a generalized drift function $h\in \Cc^1$ that satisfies the following asymptotic behavior.
	\\
	\textbf{Behavior at $-\infty$:}
	\begin{equation}\label{eq:comportement h -infty}
		\exists x_0 \leq 0,\; h_0\in \R \quad \text{such that} \quad \fa x\leq x_0,\quad h(x) = -x+h_0.
	\end{equation}
	This assumption could be generalized, this specific form allows us to perform explicit calculations which avoid some technicalities.
	\\
	\textbf{Behavior at $+\infty$:}
	\begin{equation}\label{eq:comportement h +infty}
	\begin{cases}
		\exists x_1 \geq 0,\; \fa x\geq x_1,  \quad h(x)>0, \quad h'(x)>0, 
		\\[2pt]
		\displaystyle \int_{x_1}^\infty\frac{dx}{h(x)}<\infty,
		\\[2pt]
		h(x)\sup_{y\geq x}\frac{h'(y)^2}{h(y)^4}\in L^1(x_1,\infty).
		\end{cases}
	\end{equation}
Both the quadratic and exponential models used in the literature satisfy these assumption.	Throughout this paper we use these two assumptions without necessarily mentioning them. We also use the notation, 
\begin{equation}\label{def:calV}
\mathcal{H}(x) = \int_0^x h(y)dy, \quad \text{hence} \quad  \mathcal{H}(x) \underset{-\infty}{\sim} -\frac{x^2}{2}, \quad \lim_{x\to \infty} \mathcal{H}(x) = +\infty .
\end{equation}
	The third line of assumptions \eqref{eq:comportement h +infty} ensures that 
	\begin{equation}\label{eq:limh'/h^2}
     \zeta(x):= \sup_{y\geq x}\frac{h'(y)}{h(y)^2}\in L^1(x_1,\infty) \qquad \text{and} \qquad \zeta(x) \to 0, \quad \text{as}\quad    x\rightarrow \infty.
    \end{equation}
	Indeed, we define $f(x):= h(x)\sup_{y\geq x}\frac{h'(y)^2}{h(y)^4}$ and then, for  $x\geq x_1$, $\zeta(x) = \big(\frac{f(x)}{h(x)}\big)^{1/2}$. According to assumption  \eqref{eq:comportement h +infty}, both $\frac{1}{h}$ and $f$ belong to $L^1(x_1,\infty)$ so does $\zeta$. Since it is non-increasing, the result follows.

	Such a drift $h$ has the property to send any potential that is located after $x_1$ to $+\infty$ in finite time. Indeed, ignoring the diffusion, the characteristics, $X(t)$, are determined by
	\[
%	\left\{\begin{aligned} &
		\frac{dX}{dt} = h(X(t)).
%		\\	& X(0)>x_1.
%	\end{aligned}\right.
	\]
	When  $X(0)>x_1$, the solution $X(t)$ blows-up in finite time, given by
	\[
	\lim_{t\rightarrow t_0}X(t) = +\infty, \qquad \text{with}\qquad t_0= \int_{X(0)}^\infty\frac{dy}{h(y)}.
	\]
	This kind of property is still true, with non-zero probability, when the membrane potential evolves according to the stochastic differential equation 
	\begin{equation}\label{eq:EDS IF}
	    dX_t = h(X_t)\,dt +\sqrt{2} dB_t .
	\end{equation}
In both cases, the  reset is obtained by setting $X_{t^+} = 0$ when $\lim_{t\to t^-} X_{t} =+\infty$.
%	Taking as a convention $\sigma_C = \sqrt{2}$, we get the associated Fokker-Planck equation 
%	\begin{equation}\label{eq:FPSL non lin}
%		\left\{\begin{aligned}
%			& \frac{\partial u}{\partial t}(t,x) +\frac{\partial}{\partial x}\big(\big(h(x)+bN(t) \big) u(t,x)\big) - \frac{\partial ^2u}{\partial x^2}(t,x) = \delta_{0}(x)N(t), \quad x\in \mathbb{R}, t\geq 0\\
%			& N(t) = \lim_{x\rightarrow  +\infty}h(x)u(t,x).\\
%		\end{aligned}\right.
%	\end{equation}
%	For the rest of the paper, we will only consider the linear case, by taking $b=0$. 
    The interest of the linear model lies in building a solid theoretical foundation for the existence and regularity of solutions, as well as their asymptotic behavior, on order to better address the nonlinear case.
    \\
    One of the main difficulty is to make sense of the flux $N(t)$ of particles reaching $+\infty$ when we only handle a weak solution, $u(t)\in L^1(\R)$ which is not defined pointwise.

\subsection{Main results}
	
	Our first result concerns well-posedness in $L^1$ with minimal assumptions. 
\begin{theorem} [Well-posedness in $L^1$] \label{th:existence weak sol}
	Assume that $h$ satisfies \eqref{eq:comportement h -infty} and \eqref{eq:comportement h +infty}. Then, for any initial data $u_0\in L^1(\R)$, there exists a unique weak solution $(u,N_u)\in L^\infty(\R^+;L^1(\R))\times L^1_{loc}(\R^+)$ of Eq.~\eqref{eq:FPSL lin}, i.e.,  that satisfies~\eqref{eq:formulation faible FPS} and~\eqref{eq:Conservation mass}.
 	 In addition, we have 
 	 	\begin{equation}\label{eq:L1Contract}
 	    \text{for a.e. }  t > 0, \quad \int_\R |u(t,x)|\,dx \leq  \int_\R |u_0(x)| \,dx,
 	\end{equation}
 	\begin{equation}\label{eq: U_T bound}
 	    \fa T\geq 0, \quad \sup_{x\in \mathbb{R}}\int_0^T|h(x)u(t,x)|\,dt <\infty \quad \text{and}\quad U_T = \int_0^Tu(t)\,dt \in W^{1,\infty}(\R).
 	\end{equation}
\end{theorem}
The precise definition of weak solutions, including the delicate question of the condition at infinity, that is \eqref{eq:formulation faible FPS} and~\eqref{eq:Conservation mass},  is treated in Section~\ref{sec:exist} where we also prove existence relying on a truncated equation. Uniqueness uses a regularization argument and is treated in Section~\ref{sec:unique}.
\\
	
With stronger initial data, the solution can enjoy further regularity, however limited  by the Dirac mass in the right hand side of \eqref{eq:FPSL lin}.
\begin{prop}[Regularity of the solution]\label{prop:regularity of u}.
   Let $u_0\in \Cc^0_c(\R)$, the weak solution of \eqref{eq:FPSL lin} satisfies 
	   \begin{equation}\label{eq:borne du/dx}
	\forall T>0, \qquad  \int_0^T \! \int_\R\frac{\big(\partial_xu(t,x)\big)^2}{u_\infty(x)}\,dx\,dt<\infty.
	   \end{equation}
Assume  additionally that $u_0\in H^1(\R)$, then for any $T>0$, we have $u\in L^\infty((0,T);H^1_{loc}(\R))$ and 
\[
\forall A>0, \qquad \sup_{t\in (0,T)}\sup_{|x|\leq A}|u(t+\eta, x)-u(t,x)| = O(\eta^{1/4}),
\] 
            \[
             \lim_{\eta \rightarrow 0}\sup_{t\in (0,T)}\|u(t+\eta, \cdot)-u(t,\cdot)\|_{\infty}+\|u(t+\eta, \cdot)-u(t,\cdot)\|_{1}=0.
            \]
            In particular, $u$ is continuous with respect to $x$ and $t$. 
            Finally, for all $ u_0\in L^1(\R)$, $t \mapsto u(t) \in \Cc^0\big(\R^+;L^1(\R)\big)$.
    \end{prop}
    
 This proposition, and several other regularity statements, are proved in Section~\ref{sec:regularity}. In particular it allows us to state the flux condition at infinity in a stronger form, to prove the relative entropy equality and to establish a Poincar\'e inequality for suitable forms of $h$.
 \\
 
 Although the Poincar\'e inequality implies exponential decay at infinity, the Doeblin-Harris method provides us with much more general results.
\begin{theorem} [Long-term convergence] \label{th:conv DH} 
   There are constants $M, \lambda > 0$ such that for any initial probability density $u_0$, the solution  $u(t)$  of Eq.~\eqref{eq:FPSL lin} with initial data $u_0$ satisfies
    $$
    \fa t \geq 0, \qquad \int_{\R} \big(1 + (-x)_+\big) |u(t)-u_\infty|\, dx \leq Me^{-\lambda t} \int_{\R} \big(1 + (-x)_+\big) |u_0-u_\infty | \, dx \leq +\infty, 
    $$
    where $(x)_+ = \max(x, 0)$.
\end{theorem}
The proof is detailed in Section~\ref{sec:LongTermDH}.
\\

These results open other directions which are given in the conclusion, see  Section~\ref{sec:conclusion}.

	%-------------------------------------------
	\section{Existence for the superlinear FP-IF problem}
	\label{sec:exist}
	
Our purpose is to prove the existence part of Theorem~\ref{th:existence weak sol} relying on a truncated problem. 
We begin with  a precise definition of weak solutions and give two characterizations of the condition at infinity. Then we study the steady state of Eq~\eqref{eq:FPSL lin}. It is  used to pass to the limit in the truncated problem introduced in  Section~\ref{sec:truncated}. Relying on this material, the proof of existence is completed  in Section~\ref{sec:weaksol},

%-----------------------------------------------------------
\subsection{Weak solutions of the superlinear FP-IF problem}
\label{sec:weaksol}
	
As already mentioned, a major difficulty is to define the flux at $+\infty$. We introduce it as mass conservation and give another weak sense  in Lemma~\ref{cor:lim Vu} below. In Section~\ref{sec:regularity}, with further regularity, we give a third and stronger definition.
	\begin{definition}\label{Weak sol FPS}
		We say that $(u,N_u) \in L^\infty(\R_+; L^1(\R))\times L^1_{loc}(\R_+)$ is a weak solution of Eq. \eqref{eq:FPSL lin} %with an initial data $u_0\in L^1(\R)$
		if for every test function $\varphi \in \mathcal{C}_c^\infty(\R^+\times \R)$  we have
\\[2pt]
\textbf{1. (Weak formulation of the Fokker-Planck equation)}
			\begin{align} 
				\int_0^\infty \! \int_\R u(t,x)\Big(-\partial_t\varphi(t,x)&-h(x)\partial_x\varphi(t,x)-\partial_x^2\varphi(t,x)\Big)\,dx\,dt \notag
				\\
				& = \int_0^\infty N_u(t)\varphi(t,0)\,dt
				+\int_\R u_0(x)\varphi(0,x)\,dx. \label{eq:formulation faible FPS} 
			\end{align}
\textbf{2.(Mass conservation)}
			\begin{equation}\label{eq:Conservation mass}
			\text{for } a.e. \; t\geq0 \quad 
			\int_\R u(t,x)\,dx =  \int_\R u_0(x)\,dx .
			\end{equation}
	\end{definition}
	
We would like to clarify the boundary condition at $+\infty$ in Eq.~\eqref{eq:FPSL lin}. Written in a weak form, it is in fact equivalent to the mass conservation property~\eqref{eq:Conservation mass} as we can see from the following statement.
	\begin{lemma}\label{cor:lim Vu} 
		When $u$ satisfies \eqref{eq:formulation faible FPS}, the condition \eqref{eq:Conservation mass} is equivalent to ask that for all
		$\phi\in \mathcal{C}_c^\infty(\R_+)$, and all $\chi\in \mathcal{C}_c^\infty(\R)$ with  $\int_\R\chi(x)\,dx=1$
		\begin{equation}\label{BoundaryCond}
		 \lim_{R\rightarrow \infty}\int_0^\infty\phi(t)\int_\R h(x)u(t,x)\chi(x-R)\,dx\,dt =\int_0^\infty\phi(t) N_u(t)\,dt.
		\end{equation}
	\end{lemma}
	% \begin{remark} \textcolor{blue}{Cette équivalence est toujours vraie si $N_u \in \mathcal{D}'(\R^+)$ seulement.} \end{remark}
	\begin{proof}
		 Our first goal is to define a test function that generates, in the drift term,  the function $\chi$ as used in Lemma~\ref{cor:lim Vu}. We can assume without any loss of generality that $\text{supp}(\chi)\subset [-1,1]$.
			For $R>1$, we define the function $\rho_R$ by the formula  
\begin{equation} \label{def:rhoR}	
\rho_R(x) = \int_{x}^{\infty}\big[\chi(y-R)-\frac{1}{R}\chi(\frac{y}{R}+2)\big] \,dy, \end{equation} 
and we readily check that 
\begin{equation} \label{prop:rhoR}
\text{supp}(\rho_R) \subset [-3R, R+1], \qquad  \rho_R\big|_{(-R,R-1)} = 1, \qquad \forall x \in \R, \quad  |\rho_R(x)| \leq 2 \int_\R |\chi(y)|\,dy.
\end{equation} 
		Let $(u,N_u)$ be a weak solution. For every $\phi \in \mathcal{C}_c^\infty(\R^+)$ the function $\phi(t)\rho_R(x)$ belongs to $\mathcal{C}_c^\infty(\R^+\times \R)$ and it can be used for the weak formulation \eqref{eq:formulation faible FPS}. We get
		\begin{align}
    	\int_0^\infty\bigg\{-\phi'(t)\int_\R u(t,x)\rho_R(x)\,dx 
    	- &\phi(t)\int_\R u(t,x)\big(h(x)\rho_R'(x) + \rho_R''(x)\big)\,dx\bigg\}\,dt \notag
    	\\
    	&= \underbrace{\rho_R(0)}_{=1}\int_0^\infty N_u(t)\phi(t)\,dt + \phi(0)\int_\R u_0(x)\rho_R(x)\,dx.
    	\label{eq:egality mass conservation}
    \end{align}
	Since $u(t,\cdot)$ in integrable, and by dominated convergence, as $R \to \infty$, we have
	\[
		\lim_{R\rightarrow \infty} \int_0^\infty \!\int_\R  \phi'(t) u(t,x)\rho_R(x)\,dx\, dt = \int_0^\infty\!  \int_\R \phi'(t)  u(t,x)\,dx\,dt, \quad \lim_{R\rightarrow \infty }\int_\R u_0(x) \rho_R(x)\,dx = \int_\R u_0(x)\,dx,
		\]
		\[
		\lim_{R\rightarrow \infty} \int_0^\infty \phi(t)\int_\R u(t,x) h(x) \frac{1}{R}\chi(\frac{x}{R}+2)\,dx\, dt = \lim_{R\rightarrow \infty} \int_0^\infty \phi(t)\int_{-R}^{-3R} u(t,x) \frac{|x|}{R}\chi(\frac{x}{R}+2)\,dx\, dt = 0, 
		\]
		which treats one of the terms in $h(x) \rho_R'(x)$, and for almost every $t\geq 0$,
		\begin{align*}
			\int_\R|u(t,x)\rho_R''(x)|\,dx &\leq \int_\R|u(t,x)|\big(|\chi'(x-R)|+\frac{1}{R^2}|\chi'(\frac{x}{R}+2)|\big)\,dx \to 0.
		\end{align*}
		Hence, passing to the limit as $R\rightarrow \infty$ in the above weak formulation \eqref{eq:egality mass conservation}, we get on the one hand
		\[
		\int_0^\infty\phi'(t)\int_\R u(t,x)\,dx\,dt - \phi(0)\int_\R u_0(x)\,dx = \lim_{R\rightarrow \infty}\int_0^\infty\phi(t)\int_\R h(x)u(t,x)\chi(x-R)\,dx \, dt 
		-\int_0^\infty\phi(t) N_u(t)\,dt.
		\]
		But, multiplying the mass conservation property \eqref{eq:Conservation mass} by $\phi'(t)$ and integrating, we get 
		\[
		\fa \phi \in \mathcal{C}_c^\infty(\R^+), \quad \int_0^\infty\phi'(t)\int_\R u(t,x)\,dx\,dt = -\phi(0)\int_\R u_0(x)\,dx.
		\]
		Thus we have proved the equivalence between mass conservation and the boundary condition \eqref {BoundaryCond} at~$+\infty$.
	\end{proof}
	
\subsection{The stationary state}
\label{subsection:stat state}

The stationary state plays a fundamental role in our analysis and thus we begin by studying it.

Let $(u_\infty, N_\infty)$ be a stationary state of Eq.~\eqref{eq:FPSL lin}. Because it is defined up to a multiplicative constant, we normalize it as a probability. Hence, we look for a solution of the ordinary differential equation
		\begin{equation}\label{eq:FPSL lin stat}
			\left\{\begin{aligned}
				& \big(h(x)u_\infty(x)\big)' -  u_\infty''(x) = N_\infty\delta_{0}(x), \quad x\in \mathbb{R}\\
				& N_\infty = \lim_{x\rightarrow +\infty}h(x)u_\infty(x),\qquad  \int_\R u_\infty(x)\,dx = 1.
			\end{aligned}\right.
		\end{equation}

	\begin{prop}\label{prop:u_stat}
	With assumptions \eqref{eq:comportement h -infty}, \eqref{eq:comportement h +infty} and notation \eqref{def:calV}, 
		there exists a unique solution of  \eqref{eq:FPSL lin stat}, $u_\infty\in L^1\cap L^\infty(\R)$ and $u_\infty >0$. It is given by the formula
		\begin{equation}\label{eq:uinfty_formula}
		u_\infty(x) = N_\infty e^{\mathcal{H}(x)}\int_{\max(0,x)}^\infty e^{-\mathcal{H}(y)}dy, \quad \text{with}\quad N_\infty = \bigg(\int_\R e^{\mathcal{H}(x)}\int_{\max(0,x)}^\infty e^{-\mathcal{H}(y)}dy \,dx\bigg)^{-1}.
		\end{equation}
		It is decreasing on $(x_1,\infty)$ and behaves at $\pm \infty$ as 
		\begin{equation}\label{eq:  u stat +infini}
		 u_\infty(x) =C_\infty e^{-\frac {x^2}{2}+h_0x}  \quad \text{for}\; x< x_0, \qquad    u_\infty(x) = \frac{N_\infty}{h(x)}+o\bigg(\frac{1}{h(x)}\bigg) \quad \text{as}\; x\rightarrow \infty.
		\end{equation}
        In addition, it satisfies
        \begin{equation}\label{eq:integrabilite u'}
            \int_\R\frac{(u_\infty'(x))^2}{u_\infty(x)}\,dx<\infty.
        \end{equation}
	\end{prop}
	
	\begin{proof}
		To compute $u_\infty$, we integrate equation~\eqref{eq:FPSL lin stat} between $-\infty$ and $x$ and obtain
		\begin{equation}\label{eq:ustatprime}
		\fa x \in \R, \qquad h(x)u_\infty(x) - u_\infty'(x) = N_\infty \mathbf{1}_{[0,\infty)}(x).
		\end{equation}
		Multiplying by \(e^{-\mathcal{H}(x)}\), we arrive at
		\[
		-\frac{d}{dx}\bigg(u_\infty(x)e^{-\mathcal{H}(x)}\bigg) = N_\infty \mathbf{1}_{[0,\infty)}(x)e^{-\mathcal{H}(x)}.
		\]
		Now, integrating this equation on $(x,\infty)$ and multiplying by $e^{\mathcal{H}(x)}$, we obtain the formula \eqref{eq:uinfty_formula} for $u_\infty$ after choosing $N_\infty$ to normalize $u_\infty$.
		
		According to \eqref{eq:comportement h -infty},  there is $h_1$  such that $\mathcal{H}$ is given by 
		\begin{equation}\label{eq:comportement bigV}
			\mathcal{H}(x) =
				 -\frac{1}{2}x^2+h_0x+h_1, \quad x\leq x_0.
		\end{equation}
		Therefore, the behaviour at $-\infty$ on \eqref{eq:  u stat +infini} follows from
		\[
		    u_\infty(x) = \big(e^{h_1}N_\infty \int_{0}^\infty e^{-\mathcal{H}(y)}dy\,\big)e^{-\frac {x^2}{2}+h_0x} \qquad \text{for} \quad  x< x_0.
		\]

		In addition, for $x \geq x_1$, we get thanks to definition \eqref{def:calV} and integrating by parts
		\begin{align*}
		    \int_x^\infty e^{-\mathcal{H}(y)}dy =\int_x^\infty\frac{1}{h(y)}h(y)e^{-\mathcal{H}(y)}dy& =
		    \bigg[\frac{-1}{h(y)}e^{-\mathcal{H}(y)}\bigg]_{y=x}^\infty - \int_x^\infty\frac{h'(y)}{h(y)^2}e^{-\mathcal{H}(y)}dy, \notag
		\end{align*}
    Recalling the function $\zeta(x) = \sup_{y\geq x}\frac{h'(y)}{h(y)^2}$ in \eqref{eq:limh'/h^2}, this implies
    \[
    \frac{e^{-\mathcal{H}(x)}}{h(x)} \; \frac{1}{1 +\zeta(x) }	< \int_x^\infty e^{-\mathcal{H}(y)}dy  
     < \frac{e^{-\mathcal{H}(x)}}{h(x)}, 
    \]
    and, since $1-a \leq \frac{1}{1+a}$ for $a \geq 0$, we conclude
    \begin{align} \label{est:uinfty}
    \frac{N_\infty}{h(x)} \; \Big(1 - \zeta(x)\Big) <  u_\infty (x)  
     < \frac{N_\infty}{h(x)}.
    \end{align}

   As a consequence, we obtain the asymptotic behavior of $u_\infty$ at $+\infty$.
    \\
    
    Finally, we prove \eqref{eq:integrabilite u'}. According to the behaviour of $u_\infty$ on $(-\infty,0)$, we have  $\frac{(u_\infty'(x))^2}{u_\infty(x)} = O\big(\exp(-\frac{x^2}{4})\big)$, so that it holds $\frac{(u'_\infty)^2}{u_\infty}\in L^1(-\infty,0)$.
    
   Now, for $x\geq x_1$, Eq. \eqref{eq:ustatprime} reduces to  
    \[
     u_\infty'(x) = h(x)u_\infty(x) -N_\infty .
    \]
    Thanks to estimate \eqref{est:uinfty}, we find successively 
    \[
     -N_\infty\zeta(x)<  u_\infty'(x) < 0, \qquad 
    \big(u_\infty'(x) \big)^2 < N_\infty^2 \zeta(x)^2, 
    \]
    \begin{align*}
        \frac{\big(u_\infty'(x) \big)^2}{u_\infty(x)} &\leq \underbrace{N_\infty h(x) \zeta(x)^2}_{\in L^1(x_1,\infty)}
    \underbrace{\Big( 1 - \zeta(x)\Big)^{-1}}_{=O(1)},
%    \\ &\leq N_\infty \Big( 1 + \sup_{x\geq x_1} \frac{h'(y)}{h(y)^2}\Big)h(x) \sup_{y\geq x} \frac{h'(y)^2}{h(y)^4}.
    \end{align*}
 where we use assumption \eqref{eq:comportement h +infty}. By integrability $+\infty$, the bound \eqref{eq:integrabilite u'} follows.
	\end{proof}

	\subsection{A truncated problem}
	\label{sec:truncated}
	
	Existence of a weak solution for Eq. \eqref{eq:FPSL lin} is not standard because we need to deal with the problem of the superlinear drift $h(x)$, and in particular that mass could be sent  at $+\infty$. Therefore, we consider a truncated problem for which existence is standard, see for instance \cite{AmannBook}. For a fixed $R>0$, we set
	\[
	h_R(x) = \min(h(x), h(R))\mathbf{1}_{x\geq R}+h(x)\mathbf{1}_{x< R}.
	\]
	Then, we consider the truncated equation
	\begin{equation}\label{eq:FPS trunc}
		\left\{\begin{aligned}
			&\frac{\partial u_R}{\partial t}(t,x) +\frac{\partial}{\partial x}\big(h_R(x)u_R(t,x)\big)-\frac{\partial^2u_R}{\partial x^2}(t,x) +\phi_R(x)u_R(t,x) = N_R(t)\delta_0(x), \quad x\in \R, t\geq 0\\
			&N_R(t) = \int_\R\phi_R(x) u_R(t,x)\,dx, \quad t \geq 0.
		\end{aligned}
		\right.
	\end{equation}
	With a constant $\alpha_R>0$, we define the absorption rate
	\begin{align}
	    \phi_R(x) = \alpha_R\mathbf{1}_{x\geq R}, \quad \text{with} \quad \lim_{R\rightarrow \infty}\alpha_R=+\infty, \qquad \alpha_R =o\big(h(R)^2\big), \;  R\rightarrow +\infty. \label{as:trunc}
	\end{align} 
	
	We know from methods in  \cite{AmannBook} that there is a unique solution $u_R\in \Cc^0\big(\R^+; L^1(\R)\big)$  and $N_R\in \Cc^0_b(\R_+)$. However to derive uniform bounds in $R$, we need the relative entropy inequality and thus to study the stationary states.

    \begin{prop}[Stationary state for the truncated problem] \label{Prop:N_R}
Recall \eqref{eq:comportement h -infty}--\eqref{def:calV}. For any $R>x_1$, there exists a unique stationary state $\overline{u}_R\in L^1 \cap L^\infty(\R)$ of Eq. \eqref{eq:FPS trunc} that satisfies  
      $$
       \overline{u}_R >0, \qquad \int_\R\overline{u}_R(x)\,dx = 1.
      $$
      Furthermore, with $\lambda_R = \frac{1}{2}\big(\sqrt{h(R)^2+4\alpha_R}-h(R)\big)>0$ and $\overline{N}_R>0$ a normalizing constant,   $\overline{u}_R$ is given, with $u_\infty$ constructed in Section~\ref{subsection:stat state}, by the formula
     \begin{equation}\label{eq:ecriture u_R}
       \overline{u}_R(x) = \left\{ \begin{aligned}
         &\overline{N}_Re^{\mathcal{H}(x)} \frac{e^{-\mathcal{H}(R)}}{\lambda_R+h(R)} +\frac{\overline{N}_R}{N_\infty}u_\infty(x) \qquad &\text{for}\quad x\leq R,
       \\
            &\frac{\overline{N}_R}{\lambda_R+h(R)}e^{-\lambda_R(x-R)}\quad &\text{for}\quad x\geq R, 
            \end{aligned}
            \right.
        \end{equation}
        \begin{equation}\label{eq:conv L1 u_R stat}
          \lim_{R\rightarrow\infty} \overline N_R= N_\infty, \qquad \lim_{R\rightarrow\infty}\|\overline{u}_R-u_\infty\|_\infty = 0, \qquad \text{and}\quad \lim_{R\rightarrow\infty}\|\overline{u}_R-u_\infty\|_1 = 0   .
        \end{equation}
    \end{prop}
    \begin{remark}
        One of the difficulties is that the exponential decay of $\overline{u}_R$ does not allow for a upper control of $u_\infty$ by $\overline{u}_R$. For this reason we use often the assumption $u_0\in \testfunc(\R)$. Notice that the assumption~\eqref{as:trunc} on $\alpha_R$  gives $\lambda_R \sim \frac{\alpha_R}{h(R)}$ for $R\rightarrow \infty$ and thus the mass of $\overline{u}_R$ for $x>R$ is small.
    \end{remark}

    \begin{proof}
        We first establish \eqref{eq:ecriture u_R} with $R>x_1$. A stationary solution, $\overline{u}_R$, of Eq. \eqref{eq:FPS trunc} satisfies the ordinary differential equation
        \[
        \fa x \leq R,\qquad \Big(h(x)\overline{u}_R(x)\Big)'-\overline{u}_R''(x) = \overline{N}_R\delta_0(x).
        \]
        Integrating this on $(-\infty, x]$, we obtain successively
        \[
        h(x)\overline{u}_R(x)-\overline{u}_R'(x)=\overline{N}_R\mathbf{1}_{x\geq0}, \qquad
        -\Big(\overline{u}_R(x)e^{-\mathcal{H}(x)}\Big)'=\overline{N}_Re^{-\mathcal{H}(x)}\mathbf{1}_{x\geq0}.
        \]
        Integrating on  $(x, R)$ and multiplying by \(e^{\mathcal{H}(x)}\), we get for all $x \leq R,$
        \[
         \overline{u}_R(x)=\overline{u}_R(R^-)e^{\mathcal{H}(x)-\mathcal{H}(R)}+\overline{N}_Re^{\mathcal{H}(x)}\int_{\max(0,x)}^Re^{-\mathcal{H}(y)}dy
         =\overline{u}_R(R^-)e^{\mathcal{H}(x)-\mathcal{H}(R)}+\frac{\overline{N}_R}{N_\infty} u_\infty(x)  .
        \]
        Next, on $(R,\infty)$, $\overline{u}_R$ satisfies the second-order differential equation
        \[
        \fa x >R,\quad -\overline{u}_R''(x)+h(R)\overline{u}_R'(x)+\alpha_R\overline{u}_R(x) = 0.
        \]
        The characteristic polynomial is $\lambda^2-h(R)\lambda +\alpha_R$ and $-\lambda_R$ is its unique negative root. Since we look for $\overline{u}_R\in L^1(\R)$, we get
        \[
        \fa x>R, \qquad \overline{u}_R(x)=\overline{u}_R(R^+)e^{-\lambda_R(x-R)}.
        \]
        Next, since $\overline{u}_R''$ is locally integrable in the vicinity of $R$, we need continuity for $\overline{u}_R$ and $\overline{u}_R'$
        \[
        \overline{u}_R(R^-)=\overline{u}_R(R^+)\qquad \text{and}\qquad \overline{u}_R'(R^-)=\overline{u}_R'(R^+).
        \]
        From the second equality, we get
        \[
        -\lambda_R\overline{u}_R(R^+)=h(R)\overline{u}_R(R^-)-\overline{N}_R, \qquad \text{thus} \qquad 
        \overline{u}_R(R)=\frac{\overline{N}_R}{\lambda_R+h(R)}.
        \]
        This equality shows \eqref{eq:ecriture u_R}.\\
        
        Next, we prove that $\lim_{R\rightarrow\infty}\overline{N}_R= N_\infty$. The normalizing constant $\overline{N}_R$ is given by
        $$
        (\overline{N}_R)^{-1}=  \int_{-\infty}^Re^{\mathcal{H}(x)}\frac{e^{-\mathcal{H}(R)}}{\lambda_R+h(R)}\,dx +\frac{1}{N_\infty}\int_{-\infty}^R u_\infty(x) \,dx 
        +\int_R^\infty \frac{e^{-\lambda_R(x-R)}}{\lambda_R+h(R)}\,dx.
        $$
        We treat separately the three terms on the right hand side. For the first one, we have
        \begin{align}
         \frac{e^{-\mathcal{H}(R)}}{\lambda_R+h(R)}\int_{-\infty}^Re^{\mathcal{H}(x)}\,dx
            &= \frac{e^{-\mathcal{H}(R)}}{\lambda_R+h(R)}\big(\underbrace{\int_{-\infty}^{x_1}e^{\mathcal{H}(x)}\,dx}_{=C}+\underbrace{\int_{x_1}^Re^{\mathcal{H}(x)}\,dx}_{\leq Re^{\mathcal{H}(R)}}\big) \notag
            \\
            &\leq \frac{Ce^{-\mathcal{H}(R)}+R}{\lambda_R+h(R)}\underset{R\rightarrow\infty}{\longrightarrow}0 \label{N_RFirstTerm}
        \end{align}
        because $h$ is superlinear. The second term converges to $\frac{1}{N_\infty}$ thanks to the normalization of $u_\infty$ as a probability density.
        For the last term, thanks to assumption \eqref{as:trunc}, $\alpha_R = o\big(h(R)^2\big)$ and we notice that $\lambda_R
        {\sim}\frac{\alpha_R}{h(R)}$ as $R \rightarrow +\infty$. Therefore, we may write
        \begin{align}
           0< \int_R^\infty \frac{e^{-\lambda_R(x-R)}}{\lambda_R+h(R)}\,dx = \frac{1}{\lambda_R\big(\lambda_R+h(R)\big)} \leq \frac{1}{\lambda_Rh(R)} \underset{R\rightarrow\infty}{\sim}\frac{1}{\alpha_R}\underset{R \rightarrow \infty}{\longrightarrow}0. \label{N_RThrifTerm}
        \end{align}
        This proves that $\lim_{R\rightarrow\infty}\overline{N}_R= N_\infty$.

Next, we show that \(\lim_{R\rightarrow\infty}\|\overline{u}_R-u_\infty\|_\infty=0\).     Using formula \eqref{eq:ecriture u_R}, we have 
 \begin{align*}   
     \|\overline{u}_R-u_\infty\|_\infty \leq  \frac{\overline{N}_R}{\lambda_R+h(R)}\|e^{\mathcal{H}(x)-\mathcal{H}(R)}\|_{L^\infty(-\infty, R)}+|\frac{\overline{N}_R}{N_\infty}-1| \, \|u_\infty \|_\infty + \frac{\overline{N}_R}{\lambda_R+h(R)}.
\end{align*}

        Since $\lim_{x\to -\infty}\mathcal{H}(x) = -\infty$, we have $\sup_{x \le x_1}\mathcal{H}(x) < \infty$. 
        Moreover, as $\mathcal{H}$ is increasing on $(x_1, \infty)$ by \eqref{eq:comportement h +infty} and $\lim_{R\to\infty}\mathcal{H}(R) = +\infty$, 
        we may choose $R$ large enough so that $\sup_{x \le R}\mathcal{H}(x) = \mathcal{H}(R)$. Hence, we obtain
 \begin{align*}   
     \|\overline{u}_R-u_\infty\|_\infty \leq 2 \frac{\overline{N}_R}{\lambda_R+h(R)}+ +|\frac{\overline{N}_R}{N_\infty}-1| \|u_\infty \|_\infty.
\end{align*}
Both terms tend to zero because $h(R)$ tend to infinity and we have proved that $N_R$ tends to $N_\infty$ and our statement is proved. 
\end{proof}
	
With the steady state at hand, we may state the usual relative entropy relation, \cite{RefMMP,CPSS} which is used later.

\begin{prop}[Entropy inequality for the truncated problem] \label{prop:Entropy trunc} 
	Let $\overline{u}_R$ be the stationary state of equation~\eqref{eq:FPS trunc}. Let $H :\R\rightarrow \R$ a convex function and $u_0$ an initial data such that $H(\frac{u_0}{\overline{u}_R})\overline{u}_R \in L^1(\R)$.
	Then, with $w_R = \frac{u_R}{\overline{u}_R}$, we have, for all $t\geq 0$,
\begin{align}
	\frac{d}{dt} &\int_\R  H(w_R(t,x))\overline{u_R}(x)\,dx =-\int_\R\left(\partial_x w_R(t,x)\right)^2 H''(w_R(t,x)) \overline{u}_R(x)\,dx \label{eq:Entropie trunc}
	\\
	& -\int_\R \phi_R(x) \big\{H(w_R(t,x))-H(w_R(t,0))-H'(w_R(t,0))(w_R(t,x)-w_R(t,0))\big\}
	\frac{\overline{u}_R(x)}{\overline{N}_R(x)}\,dx \leq 0.  \notag
\end{align}
\end{prop}

	\subsection{Proof of existence of a solution}
	\label{sec:proof1}
		
	Letting $R\to \infty$ in \eqref{eq:FPS trunc}, we can can prove that $u_R$ converges in a weak sense to a solution of the weak formulation of~\eqref{eq:FPSL lin} and thus establish Theorem~\ref{th:existence weak sol}.
	
	\begin{proof}
	\noindent {\em Step 1.} \textbf{Estimates for the weak solution $(u,N)$.}\\
	Let $u_0 \in L^1(\R)$ and $(u_R(t))_{R>x_1}$ the solutions associated to the truncated problem \eqref{eq:FPS trunc}.
	Using the entropy inequality from Proposition \ref{prop:Entropy trunc}, with the convex function $H(x) = |x|$, we get 
	\begin{equation}\label{eq:abs_u_R}
    \partial_t|u_R(t,x)|+\partial_x\big(h_R(x)|u_R(t,x)|\big)-\partial_x^2|u_R(t,x)|+\phi_R(x)|u_R(t,x)|\leq |N_R(t)|\delta_0, 
    \end{equation}
	\begin{equation}\label{eq:est u_R}
	    \fa R>x_1, \fa t \geq 0, \qquad \int_\R|u_R(t,x)|\,dx\leq \int_\R|u_0(x)|\,dx.
	\end{equation}

    Our first goal is to prove that for any $T>0$ and a constant $c_T>0$, we have
    \begin{equation}\label{eq:est N_R}
         \sup_{R>x_1} \int_0^T|N_R(t)| dt \leq c_T\|u_0\|_1.
    \end{equation}
     We consider a function  $\chi_1\in\mathcal{C}^\infty(\R,\R^+)$  such that $\text{supp}(\chi_1)\subset (0,\infty)$ and $\chi_1 \big|_{(1,\infty)} = 1$. Using the bound $u_R\leq C\overline{u}_R$, and integrating by parts, we get from Eq.~\eqref{eq:abs_u_R}
    \[
    \frac{d}{dt}\int_\R\chi_1(x)|u_R(t,x)|\,dx+\int_\R|u_R(t,x)|\big(h_R(x)\chi_1'(x)+\chi_1''(x)\big)\,dx+\int_\R \chi_1(x)\phi_R(x)|u_R(t,x)|\,dx \leq 0.
    \]
    Now, for $R>1$, we have $\chi_1\phi_R =\phi_R$ and $\chi_1'h_R = \chi'_1V$ which gives
        $$
        \int_\R|u_R(t)|\big(h_R\chi_1'+\chi_1''\big)\,dx\geq -\|V\chi_1'+\chi_1''\|_\infty\|u_R(t)\|_1\geq -\|V\chi_1'+\chi_1''\|_\infty\|u_0\|_1.
        $$
        Also, for all  $t \geq0$, it holds $|N_R(t)|\leq \int_\R\phi_R(x)|u_R(t,x)|\,dx$ and we arrive at  the inequality
        \[
        \fa t \geq 0, \quad \frac{d}{dt}\int_\R\chi_1(x)|u_R(t,x)|\,dx-\|V\chi_1'+\chi_1''\|_\infty\|u_0\|_1+|N_R(t)|\leq 0.
        \]
        Integrating this on $[0,T]$ for any $T>0$, we get
        \begin{align*}
        \int_0^T |N_R(t)| dt 
        &\leq \int_\mathbb{R} \chi_1(x)\big(u_0(x) - u_R(t,x)\big)\,dx 
            + \|V\chi_1' + \chi_1''\|_\infty \|u_0\|_1 T \notag \\
        &\leq \big(1 + \|V\chi_1' + \chi_1''\|_\infty T\big) \|u_0\|_1= c_T \|u_0\|_1,
    \end{align*}
    and \eqref{eq:est N_R} is proved.
    \\
    
	Next, we assume $u_0\in  \Cc^0_c(\R)$, then, we can find a uniform constant $C_0 >0$ such that
	  \begin{equation} \label{as:IDcompactSupp}
	   \forall  R>x_1, \qquad  |u_0|\leq C_0 \; \overline{u}_R.
	  \end{equation}
	  Then, we also have for every $x \in \mathbb{R}$ and $t \ge 0$, 
	    \begin{align} 
	        |u_R(t,x)| \le C_0 \overline{u}_R(x), 
	     \qquad|N_R(t)|\leq C_0 \int_R^\infty\phi_R(x)\overline{u}_R(x)\,dx \leq C_0 \; \sup_{R>x_1}\overline{N}_R. \label{est:uR}
	     \end{align}
	Indeed, using the convex function $H(x) =\max(0,|x|-C_0)$ and Proposition \ref{prop:Entropy trunc}, we find 
	    \[
	    0\leq \int_\R H\bigg(\frac{u_R(t,x)}{\overline{u}_R(x)}\bigg)\overline{u}_R(x)\,dx \leq \int_\R H\bigg(\frac{u_0(x)}{\overline{u}_R(x)}\bigg)\overline{u}_R(x)\,dx = 0.
	    \]

	    \noindent {\em Step 2.} \textbf{Existence of the weak solution for  $u_0\in \Cc^0_c(\R)$.}
	    \\
	     It is also convenient to use the function $w_R(t) = \frac{u_R(t)}{\overline{u}_R}$, then, still assuming \eqref{as:IDcompactSupp} and according to~\eqref{est:uR}, 
	    \begin{equation}\label{est:wR}
	        \sup_{R>x_1}\sup_{t\geq 0}\max\bigg(\|w_R(t)\|_\infty, \frac{|N_R(t)|}{\overline{N}_R}\bigg)\leq C_0.
	    \end{equation}
	    Thanks to these uniform bounds, we can extract a subsequence $(R_n)_{n\in \N}$ such that $R_n\rightarrow\infty$ and we can find $(u,w,N)\in L^\infty(\R^+\times \R)\times L^\infty(\R^+\times \R)\times L^\infty(\R^+)$ such that, in the $L^\infty$ weak-* topology, see~\cite{brezis2010functional}, 
	    \begin{equation} \label{bd:uwNcomp}
	    u_{R_n}  \rightharpoonup u, \qquad
	    w_{R_n} \rightharpoonup w=\frac{u}{u_\infty} \leq C_0, \qquad \text{and}\qquad N_{R_n}\rightharpoonup N\leq C_0 \quad \text{as}\quad n\rightarrow \infty.
	    \end{equation}
	    We also immediately deduce, using \eqref{eq:est u_R} the $L^1$ contraction property
	    \[
	    u=w u_\infty \in  L^\infty \big(\R^+; L^1\cap L^\infty (\R)\big), \qquad \int_\R |u(t,x)| \, dx \leq \int_\R |u_0(x)| \, dx\quad \text{for a.e. } t>0.
	    \]

Moreover, $(u,N)$ satisfy \eqref{eq:FPSL lin} in the weak sense. Indeed, firstly, the weak formulation of Eq. \eqref{eq:FPS trunc} gives that for any $\varphi\in \testfunc(\R^+\times \R)$ and $n\in \N$,
	    \begin{align*}
	       \int_0^\infty\int_\R u_{R_n}(t,x)\left(-\partial_t\varphi(t,x)-V_{R_n}(x)\partial_x\varphi(t,x)-\partial_x^2\varphi(t,x)-\phi_{R_n}(x)\varphi(t,x)\right)\,dx\,dt\\
	       = \int_0^\infty N_{R_n}(t)\varphi(t,0)\,dt
		+\int_\R u_0(x)\varphi(0,x)\,dx .
	    \end{align*}
    Taking the limit as $n\rightarrow \infty$, since $\varphi$ has a compact support and supp$(\phi_{R_n})\subset (R_n,\infty)$, then 
    \[
    \lim_{n\rightarrow\infty} \int_0^\infty\int_\R u_{R_n}(t,x)\phi_{R_n}(x)\varphi(t,x)\,dx\,dt = 0.
    \]
    Therefore we have obtained the first item  \eqref{eq:formulation faible FPS}  of the weak formulation of the Fokker-Planck Eq.~\eqref{eq:FPSL lin}.

	Secondly, we also have the mass conservation property \eqref{eq:Conservation mass} for $u$. Indeed, using the mass conservation property of the family $(u_{R_n})_{n\geq 0}$, we have for any $t\geq 0$ and $A>0$,
    \[
	\int_{\R} u_0(x)\,dx = \int_\R u_{R_n}(t,x)\,dx =\int_{-A}^A u_{R_n}(t,x)\,dx + \int_{|x|\geq A}u_{R_n}(t,x)\,dx ,
	\]
    so for any $A>0$, we have
    \[
    \bigg|\int_{-A}^Au(t,x)\,dx -\int_\R u_0(x)\,dx\bigg|\leq \bigg|\int_{-A}^Au(t,x)\,dx -\int_{-A}^A u_{R_n}(x)\,dx\bigg|+\bigg|\int_{|x|\geq A}u_{R_n}(t,x)\,dx\bigg|.
    \]
    Using the weak convergence of $u_{R_n}$, then, 
    \[
    \lim_{n\rightarrow\infty}\bigg|\int_{-A}^Au(t,x)\,dx -\int_{-A}^A u_{R_n}(x)\,dx\bigg|=0
    \]
    and 
    \[
    \limsup_{n\rightarrow\infty}\bigg|\int_{|x|\geq A}u_{R_n}(t,x)\,dx\bigg|\leq C\limsup_{n\rightarrow\infty}\int_{|x|\geq A}\overline{u}_{R_n}(x)\,dx = C\int_{|x|\geq A}u_\infty(x)\,dx .
    \]
    So for every $A>0$, 
    \[
    \bigg|\int_{-A}^Au(t,x)\,dx -\int_\R u_0(x)\,dx\bigg|\leq C\int_{|x|\geq A}u_\infty(x)\,dx \underset{A\rightarrow\infty}{\longrightarrow}0.
    \]
    Since $u \in L^\infty(\R^+; L^1(\R))$, we also have $\int_{-A}^A u(t,x)\,dx \to \int_{\R} u(t,x)\,dx$ and mass conservation is proved.

	Hence, the functions $(u,N_u)$ satisfies the condition \eqref{eq:formulation faible FPS} and \eqref{eq:Conservation mass}, which proves the existence of a weak solution for any $u_0\in \testfunc(\R)$.
	\\

    \noindent {\em Step 3.} \textbf{Existence of the weak solution for $u_0\in L^1(\R)$.}\\
    
    We extend the existence result to initial data satisfying $u_0\in L^1(\R)$ based on the $L^1$ contraction property.
    \\
    
    Let $(u_0^{(k)})_{k\in \N}\in \testfunc(\R)^\N$ such that $\lim_{k\rightarrow\infty}\|u_0^{(k)}-u_0\|_1=0$. 
    For every $R>0$, let $(u_R^{(k)},N_R^{(k)})$ be the solution to the  truncated problem with data $u_0^{(k)}$. 
    
    By the uniqueness of the solution to the truncated problem, $\fa k,l\in \N$, $u_R^{(k)}-u_R^{(l)}$ is the solution to the  truncated problem with data $u_0^{(k)}-u_0^{(l)}$. The $L^1$ non-expansion properties \eqref{eq:est u_R}--\eqref{eq:est N_R} gives for every $k,l\in \N$ and all $t \geq 0$, 
    \[
    \|u_{R}^{(k)}(t)-u_{R}^{(l)}(t) \|_1 \leq \|u_{0}^{(k)}-u_{0}^{(k)}\|_1 .
    \]
    
    Now, using the result of step 2 and either the uniqueness of the limit which is proved independently later, or the Cantor diagonal argument, we can find a subsequence $(R_n)_{n\in \N}$ and a sequence  $(u^{(k)},N^{(k)})_{k\in \N}$  such that for all $k\in \N$, $u_{R_n}^{(k)}\underset{n\rightarrow\infty}{\rightharpoonup}u^{(k)}$ and $N_{R_n}^{(k)}\underset{n\rightarrow\infty}{\rightharpoonup}N^{(k)}$.
    
    These functions $(u^{(k)},N^{(k)})$ are the weak solutions  solution of \eqref{eq:FPSL lin} corresponding to the data $u_0^{(k)}$ still satisfying \eqref{bd:uwNcomp} with a constant $C_0^{(k)}$. The $L^1$ non-expansion properties holds for every $k,l\in \N$, for all $t\geq 0$ and $T\geq 0$, 
    \[
    \sup_{t\geq0}\|u^{(k)}(t)-u^{(l)}(t)\|_1\leq \|u_0^{(k)}-u_0^{(l)}\|_1\qquad \text{and}\qquad \int_0^T|N^{(k)}(t)-N^{(l)}(t)|dt\leq c_T\|u_0^{(k)}-u_0^{(l)}\|_1.
    \]
    This inequality shows that $(u^{(k)})_{k\in \N}$ is a Cauchy sequence in $L^\infty(\R^+;L^1(\R))$ and thus it converges to a limit $u\in L^\infty(\R^+;L^1(\R))$. Similarly, for any $T>0$, $(N^{(k)})_{k\in \N}$ is a Cauchy sequence in $L^1(0,T)$ so it admits a limit $N_T\in L^1(0,T)$. Finally, by uniqueness of the limit on every interval $(0,T)$, we can define the function $N$ in $(0,\infty)$. The function $N$ satisfies $N\in L^1_{loc}(\R^+)$ and for all $ T\geq 0$, we have $\lim_{n\rightarrow \infty}\int_0^T|N(t)-N^{(k)}(t)|dt =0$.
    \\
    
    Since ($u^{(k)},N^{(k)})$ is a weak solution for every $k\in \N$, it satisfies the weak formulation \eqref{eq:formulation faible FPS} for any $ \varphi\in \testfunc(\R^+\times\R)$
    \[
    \int_0^\infty\int_\R u^{(k)}(t,x)\left(-\partial_t\varphi(t,x)-h(x)\partial_x\varphi(t,x)-\partial_x^2\varphi(t,x)\right)\,dx = \int_0^\infty N^{(k)}(t)\varphi(t,0)\,dt
				+\int_\R u_0^{(k)}(x)\varphi(0,x)\,dx
    \]
    and,
    \[
    \fa t \geq 0, \qquad \int_\R u^{(k)}(t,x)\,dx\,dt = \int_\R u_0^{(k)}(x)\,dx.
    \]
    Taking the limit as $k\rightarrow \infty$ in both equations, we finally obtain
    \[
    \int_\R u(t,x)\left(-\partial_t\varphi(t,x)-h(x)\partial_x\varphi(t,x)-\partial_x^2\varphi(t,x)\right)\,dx\,dt = \int_0^\infty N(t)\varphi(t,0)\,dt
				+\int_\R u_0(x)\varphi(0,x)\,dx
    \]
    and
    \[
    \fa t\geq 0, \quad \int_\R u(t,x)\,dx\,dt = \int_\R u_0(x)\,dx.
    \]

    Hence, the functions $(u,N)$ is a weak solution of Eq. \eqref{eq:FPSL lin} since \eqref{eq:formulation faible FPS} holds and it satisfies the mass conservation condition \eqref{eq:Conservation mass}. This  concludes the proof of theorem~\ref{th:existence weak sol} for the existence of a weak solution for $u_0\in L^1(\R)$. 
	\end{proof}
        
      \noindent {\em Step 4.} \textbf{Proof of estimate  \eqref{eq: U_T bound}.}
    We may remove the absolute value in \eqref{eq: U_T bound} by proving the result for nonnegative initial data. Indeed, since \eqref{eq:FPSL lin} preserves positivity, we have $|u| \leq u_+ + u_-$ where $u_\pm$ are the solutions corresponding to the initial data $\max(0, u_0)$ and $\max(0, -u_0)$.

    Let $T>0$ and define  
    \[
    U_T(x)=\int_0^T u(t,x)\,dt .
    \]
    Integrating Eq.~\eqref{eq:FPSL lin} over $(0,T)$, we obtain,
    \[
    u(T,x)-u_0(x)+\partial_x\Big(h(x)U_T(x)\Big) -U''_T(x)
    =\delta_0\int_0^T N(t)\,dt.
    \]
    This identity is understood in the sense of distributions on $\R$. Integrating it on $(x,\infty)$ yields
    \begin{equation}\label{eq: U_T}
        -U_T'(x)+h(x)U_T(x)
    =\mathbf{1}_{x\geq 0}\int_0^T N(t)\,dt
    +\int_x^\infty \big(u(T,y)-u_0(y)\big)\,dy =: F_T(x).
    \end{equation}
    Then $F_T\in L^\infty(\R)$ and
    \[
    \|F_T\|_\infty\leq \int_0^T |N(t)|\,dt + \|u_0\|_1 + \|u(T)\|_1 .
    \]
    
    Multiplying the previous identity by $e^{-\mathcal{H}(x)}$ and integrating over $(x,\infty)$, we obtain for all $x\in\R$,
    \[
    U_T(x)e^{-\mathcal{H}(x)} =\int_x^\infty e^{-\mathcal{H}(y)} F_T(y)\,dy .
    \]
   Define
    \[
    G(x)=|h(x)| e^{\mathcal{H}(x)}\int_x^\infty e^{-\mathcal{H}(y)}\,dy =\frac{|h(x)|u_\infty(x)}{N_\infty},
    \]
    where we used the formula \eqref{eq:uinfty_formula}. Then, we have 
    \begin{equation*}
        \sup_{x\in\R} |h(x)| U_T(x)
    \leq \|F_T\|_\infty
    \,\sup_{x\in\R}G(x) \qquad \text{and}\quad \|U_T\|_\infty \leq \|F_T\|_\infty  \sup_{x\in\R}\frac{u_\infty(x)}{N_\infty}.
    \end{equation*}
    From the estimates \eqref{eq:  u stat +infini}, $u_\infty$ and $h\, u_\infty$ are bounded and thus $U_T$ and $h(x)|U_T(x)|$ are  bounded on~$\R$. From Eq.~\eqref{eq: U_T} we also conclude that $U_T'$ is bounded.
    
    This concludes the proof of \eqref{eq: U_T bound} and of the existence part of 
    Theorem~\ref{th:existence weak sol}.

	\section{Uniqueness of the weak solutions}
	\label{sec:unique}
	
	We now prove the uniqueness of the weak solution of Eq.~\eqref{eq:FPSL lin} as stated in Theorem~\ref{th:existence weak sol}.  Our proof uses  a regularization argument, therefore we consider a regularizing kernel~$\omega$  with the properties
	\[
	\omega \in \mathcal{C}^2(\R),\quad \text{supp}\, (\omega )\subset (-1,1), \quad \int_\R\omega(x)\,dx = 1,\quad \text{and}\qquad  \forall x\in \R, \quad  \omega(x) \geq 0, \quad -x\omega'(x)\geq 0.
	\]
	 For $\varepsilon >0$, we define the scaled function $\omega_\varepsilon(x) = \frac{1}{\varepsilon}\omega\big(\frac{x}{\varepsilon}\big)$.
	This allows us to define the convolution 
	\[
	u_\varepsilon (t,x) := u(t)\ast\omega_\varepsilon(x) = \int_\R u(t,y)\omega_\varepsilon(x-y)dy \in L^\infty\big(\R^+; L^\infty(\R)\cap L^1(\R)\big). 
	\]
As it is standard, see \cite{DiPL89, Lebris_Lions_2008, Figalli2008}, \( u_\varepsilon \) satisfies the regularized equation for all $t\geq 0$ and $x\in\R$, 
			\begin{align}\label{SLIF:regul}
			\begin{cases}
			\partial_tu_\varepsilon(t,x) +\partial_x\big(h(x)u_\varepsilon(t,x) \big)-\partial_x^2 u_\varepsilon(t,x) +R_\varepsilon(t,x)= N_u(t)\omega_\varepsilon(x),
			\\[5pt]
			R_\varepsilon(t,x) = \partial_x\big((Vu(t))\ast\omega_\varepsilon(x) -h(x)u_\varepsilon(t,x)\big),
			\end{cases}
			\end{align}
with initial data $\ureg^0 = u_0\ast\reg$.

Notice that \(u_\eps\) is $\mathcal{C}^\infty$ in $x$ but only $W_{loc}^{1,1}$ in $t$ and thus it is continuous in time.

	\begin{proof}
		It is enough to show that a weak solution $(u,N_u)$ with $u_0=0$ vanishes for all $t\geq 0$.
		
		The regularized solution $\ureg$ satisfies Eq.~\eqref{SLIF:regul} and by Kato's inequality, we get in distributional sense
		\begin{align*}
		\partial_t |\ureg(t,x)|+\partial_x\left(h(x)|\ureg(t,x)|\right)-\partial_x^2|\ureg(t,x)|+\signe(\ureg(t,x))R_\varepsilon(t,x)& \leq N_u(t)\signe(\ureg(t,x))\reg(x)
		\\
		&\leq |N_u(t)|\reg(x).
		\end{align*}
		For $\chi \in \mathcal{C}_c^\infty(\R)$ such that $\chi \geq 0$, with supp($\chi) \subset [-1,1],$ we define $\rho_R\geq 0$ as in \eqref{def:rhoR} and recall the properties \eqref{prop:rhoR}. Multiplying the inequality by $\rho_R(x)$, integrating over $\R$, and integrating
		by parts for the spatial derivative terms, we obtain
		\begin{align*}
			\frac{d}{dt}\int_\R\rho_R(x)|\ureg(t,x)|\,dx &+\int_\R|\ureg(t,x)|(-h\rho_R'(x)-\rho_R''(x))\,dx +\int_\R\rho_R(x)\signe(\ureg(t,x))R_\varepsilon(x,t)\,dx\\
			&\leq |N_u(t)|\int_\R\rho_R(x)\reg(x)\,dx\\
			&\leq |N_u(t)| \quad \big(\text{since}\:\:  \int_\R\reg(x)\,dx=1\:\:\text{and}\:\:\rho_R\big|_{(-\varepsilon,\varepsilon)}=1\big).
		\end{align*}
		
		Now, we prove that the term $R_\varepsilon$ converges to $0$ in $L^1_{loc}$. For any $ t\geq 0,$ and $x \in \R,$ we may write
		\begin{align*}
			R_\varepsilon(t,x) 
			&= \big(hu(t)\big)\ast\reg'(x) -h(x)u(t)\ast \reg'(x) - h'(x)u(t)\ast\reg(x)\\
			&=\int_\R\bigg\{\big(h(y)-h(x)\big)\reg'(x-y) - h'(x)\reg(x-y)\bigg\}u(t,y)dy\\
			&= \int_\R\bigg\{\int_x^y h'(z)dz\,\reg'(x-y) - h'(x)\reg(x-y)\bigg\}u(t,y)dy\\
			&= \int_\R\bigg\{\frac{1}{y-x}\int_x^y\big(h'(z)-h'(x)\big)dz\,\mu_\varepsilon(x-y)+h'(x)\big(\mu_\varepsilon(x-y)-\reg(x-y)\big)\bigg\}u(t,y)dy
		\end{align*}
		where $\mu_\varepsilon(z) = -z\reg'(z)$. By hypothesis on $\omega$, $\mu_\varepsilon$ satisfies $ \mu_\varepsilon(-x) = \mu_\varepsilon(x)$ and $\mu_\varepsilon \geq 0$. In addition, $\int_\R\mu_\varepsilon(x)\,dx = \int_\R-x\reg'(x)\,dx = \int_\R\reg(x)\,dx = 1$. Therefore, the family $(\mu_\varepsilon)_{\varepsilon>0}$ is also an approximation of the identity.\\
		
		Now, we use $x\in$ supp$(\rho_R) \subset (-3R,R+1)$ and $y\in \R$ such that $|x-y|<\varepsilon$, we have
		\[
		\bigg|\frac{1}{y-x}\int_x^y\big(h'(z)-h'(x)\big)dz\bigg|\leq \sup_{z\in[x,y]}|h''(z)|\,|x-y|\leq \|h''\|_{\infty,(-3R-\varepsilon,R+1+\varepsilon)}|x-y|.
		\]
		Hence, for $\varepsilon\in (0,1)$, we get the following bound:
		\begin{align*}
			|R_\varepsilon(x,t)|&\leq \|h''\|_{\infty,[-3R-1,R+2]}\int_\R|x-y|\mu_\varepsilon(x-y)|u(t,y)|dy +|h'(x)||u(t)\ast \mu_\varepsilon(x)-u(t)\ast \reg(x)|\\
			&\leq \varepsilon\|h''\|_{\infty,[-3R-1,R+2]}\int_\R \mu_\varepsilon(x-y)|u(t,y)|dy+\|h'\|_{\infty,[-3R-1,R+2]}|u(t)\ast \mu_\varepsilon(x)-u(t)\ast \reg(x)|.
		\end{align*}

		Then, we may write
		\[
		\int_\R\rho_R(x)|R_\varepsilon(x,t)|\,dx\leq \varepsilon\|h''\|_{\infty,[-3R-1,R+2]}  \|u(t)\|_1+\|h'\|_{\infty,[-3R-1,R+2]}\|u(t)\ast(\mu_\varepsilon- \reg) \|_{1}
		\]
	    Finally, since $u\in L^\infty(\R^+;L^1(\R))$, for almost every $t\geq 0,\lim_{\varepsilon\rightarrow 0}\|u(t)\ast (\mu_\varepsilon-\reg)\|_1 = 0$ and $\|u(t)\ast (\mu_\varepsilon-\reg) \|_1\leq 2\|u(t)\|_1$. Using the dominated convergence theorem, we get
		\begin{align} \label{cvg:Remainder}
		\fa \phi \in \testfunc(\R^+), \quad \lim_{\varepsilon\rightarrow 0}\int_0^\infty\phi(t) \int_\R\rho_R(x)|R_\varepsilon(x,t)|\,dx \,dt = 0,
			\end{align}
		as announced.
		\\

		We may now proceed to prove that $u$ vanishes.  Taking the limit as $\varepsilon \rightarrow 0$,  we get for $\phi \in \testfunc(\R^+)$,
		\begin{align}
			-\int_0^\infty \! \int_\R  \phi'(t) \rho_R(x) |u(t,x)|\,dx\,dt +\int_0^\infty\! \int_\R & \phi(t) |u(t,x)|\big(-h(x)\rho_R'(x)-\rho_R''(x)\big)\,dx\,dt \notag
			\\
			&\leq \int_0^\infty \phi(t)|N_u(t)|dt. \label{eq:|u|}
		\end{align}
		In this inequality, several terms tend to $0$ when $R \to \infty$. We analyze them now. By construction of the function $\rho_R$, we have
		\[
		-\int_\R|u(t,x)|h(x)\rho'_R(x)\,dx = \int_\R|u(t,x)|h(x)\big(\chi(x-R)-\frac{1}{R}\chi(\frac{x}{R}+2)\big)\,dx.
		\]
	    For the second term, we write, for any $R>1$
		\[
		\frac{1}{R}\int_\R|h(x)u(t,x)|\chi(\frac{x}{R}+2)\,dx\leq \frac{|h(-3R)|\|\chi\|_\infty}{R}\int_{-\infty}^{-R}|u(t,x)|\,dx  \underset{R\rightarrow\infty} \longrightarrow 0,
		\]
		where we use that $h(x) = O(|x|)$ according to~\eqref{eq:comportement h -infty} when $x\rightarrow-\infty$ and that $u(t)\in L^1(\R)$.
		
		Furthermore, we have that, for all $t\geq 0$,
		\[
	   \int_\R|u(t,x)\rho''_R(x)|\,dx\leq \|\chi'\|_\infty\bigg(\int_{R-1}^\infty|u(t,x)|\,dx+\frac{1}{R^2}\int_{-\infty}^{-R+1}|u(t,x)\,dx|\bigg)\rightarrow 0\:\: \text{when}\:\: R\rightarrow \infty.
		\]
		
	   With these two terms vanishing, taking the limit in \eqref{eq:|u|} when $R\rightarrow \infty$, we get
		$$
		\int_0^\infty-\phi'(t)\int_\R|u(t,x)|\,dx\,dt \leq \int_0^\infty\phi(t)\bigg\{|N_u(t)| - \limsup_{R\rightarrow \infty}\int_\R h(x)|u(t,x)|\chi(x-R)\,dx\bigg\}\,dt.
		$$
		Finally, since $u$ satisfies the mass conservation property and according to Lemma~\ref{cor:lim Vu}, we know that 
		$$
		\int_0^\infty\phi(t)\bigg\{N_u(t) - \lim_{R\rightarrow \infty}\int_\R h(x)u(t,x)\chi(x-R)\,dx\bigg\}\,dt = 0.
		$$
		Using the property of weak convergence, we obtain
		\[
		 \int_0^\infty\phi(t)\bigg\{|N_u(t)| - \limsup_{R\rightarrow \infty}\int_\R h(x)|u(t,x)|\chi(x-R)\,dx\bigg\}\,dt\leq 0.
		\]
		We conclude that
		$$
		\int_0^\infty-\phi'(t)\int_\R|u(t,x)|\,dx\,dt \leq 0.
		$$
		Since $\phi$ is an arbitrary non-negative test function, this implies that $\frac{d}{dt} \int_\R |u(t,x)|\,dx \leq 0$ in the sense of distributions. Given the data $u(t=0) = 0$, it must be that for a.e. $t\geq 0$, $\int_\R|u(t,x)|\,dx = 0$. This implies that $u=0$, thus proving the uniqueness of the solution.
	\end{proof}

	\section{Regularity of the solution and entropy inequality}
	\label{sec:regularity}

	So far,  we have built weak solutions of Eq.~\eqref{eq:FPSL lin} in Theorem~\ref{th:existence weak sol} which are merely $L^1$. We now establish further regularity and relative entropy dissipation.
	
%-----------------------
\subsection{Maximum principle}
%-----------------------
	
\begin{prop}[Regularity of the solution $u$]\label{prop:MaxPrinc} Assume that for some constant $|u_0| \leq C_0 u_\infty$, then the weak solution of \eqref{eq:FPSL lin} satisfies for all $t\geq 0$, $| u(t)| \leq C_0 u_\infty $, $N(t) \leq C_0 N_\infty$.
\end{prop}

\begin{proof} 
The solution has been built as the limit in $C\big( [0,T]; L^1(\R)\big)$ of a sequence $u^{(k)}$ corresponding to initial data $u_0^{(k)}\in \Cc^0_c(\R)$ and we may always assume that $| u_0^{(k)} | \leq C_0 u_\infty$. Each $u^{(k)}$ is itself the limit of a sequence $u_{R_n}^{(k)}$ of solutions to the truncated problem~\eqref{eq:FPS trunc} with initial data $u_0^{(k)}$.  

Let $\alpha >1$. By the uniform convergence of the steady state ${\overline  u}_R$ to $u_\infty$, see~\eqref{eq:conv L1 u_R stat}, and because $u_\infty (x) >0$ for all $x \in \R$ and $u_0^{(k)}$ has compact support, we have $|u_0^{(k)} | \leq (1+\alpha) C_0 {\overline  u}_R$ for $R$ large enough as already used in~\eqref{est:wR}. As a consequence, we also have $| u_{R_n}^{(k)}| \leq (1+\alpha) C_0 {\overline  u}_{R_n}$ and thus, in the limit, $| u^{(k)}| \leq (1+\alpha) C^0 {u}_\infty$. Since this holds for all $\alpha >0$, we also have $| u^{(k)}| \leq C_0 {u}_\infty$ and in the limit we obtain the announced result.
\end{proof}    

	%-----------------------
	\subsection{Regularity}
	%-----------------------
	
    We now prove our main regularity result, namely Proposition~\ref{prop:regularity of u}.
    
    \begin{proof}
    Let $u$ be the weak solution of \eqref{eq:FPSL lin} with initial data $u_0\in \Cc^0_c(\R)$. We set $w(t)=\frac{u(t)}{u_\infty}$.\\
    
    \noindent {\em Step 1. We show that $u, \,w \in L^2((0,T);H^1_{loc}(\R))$.} 
    For $R>x_1$, let $u_R(t)$ be the solution to the  truncated problem \eqref{eq:FPS trunc} with initial data $u_0$. As in the proof of Theorem \ref{th:existence weak sol}, using that $u_0\in \Cc^0_c(\R)$, we can find a uniform constant $C_0>0$ such that for all $R>x_1$, $|u_0| \leq C_0 \overline{u}_R$. That implies that $w_R(t) = \frac{u_R(t)}{\overline{u}_R}$ satisfies $|w_R(t) | \leq C_0$ and also $|N_R(t)| \leq C_0 $. 
    
    Choosing the convex function $H(z) = \frac{z^2}{2}$ in \eqref{eq:Entropie trunc} and  integrating, we get
        \[
        \int_0^\infty\int_\R \big(\partial_xw_R(t,x)\big)^2\overline{u}_R(x)\,dx\,dt\leq \int_\R \left(\frac{u_0(x)}{\overline{u}_R(x)}\right)^2\overline{u}_R(x)\,dx \leq C_0^2.
        \]
        As in the step 1 of the proof of Theorem \ref{th:existence weak sol}, and using uniqueness of the limit,  $w_{R}\rightharpoonup w$ as $R \rightarrow \infty$ in the $L^\infty$-w* topology. From the above inequality, for all $A>0$, we have 
        \[
        \int_0^\infty\int_A^A \big(\partial_x w_R(t,x)\big)^2\,dx\,dt\leq C_0^2 \min_{[-A,A]}(\overline{u}_R)^{-1}.
        \]
        As a consequence, \( \partial_x w_R \rightharpoonup \partial_x w \) weakly in $L^2 \big((0,T)\times(-A,A)\big)$. By weak-strong limit, since $\overline{u}_R\rightarrow u_\infty$ strongly in $\Cc^0(\R)$, we get
        \[
        \int_0^\infty\int_{-A}^A\big(\partial_x w(t,x)\big)^2 u_\infty(x)\,dx\,dt\leq \liminf_{R\rightarrow\infty}\int_0^\infty\int_{-A}^A \big(\partial_xw_R(t,x)\big)^2\overline{u}_R(x)\,dx\,dt\leq C_0^2.
        \]
        Because this is true for all $A>0$, we also have 
        \begin{equation}\label{eq: borne d_x w }
            \int_0^\infty\int_\R \big(\partial_x w(t,x)\big)^2 u_\infty(x)\,dx\,dt\leq C_0^2.
        \end{equation}
        Finally since $ u_\infty$ is smooth, and $w\in L^\infty(\R^+\times \R)$, we  obtain
        \[
        \frac{\big(\partial_xu(t,x)\big)^2}{u_\infty(x)} =\frac{\big(w(t,x)u_\infty'(x)+u_\infty(x)\partial_xw(t,x)\big)^2}{u_\infty(x)}\leq w(t,x)^2\frac{(u'_\infty(x))^2}{u_\infty(x)}+u_\infty(x)\big(\partial_xw(t,x)\big)^2.
        \]
        Using that $w\in L^\infty(\R^+\times \R)$ and $\frac{(u_\infty')^2}{u_\infty}\in L^1(\R)$ according to Proposition \ref{prop:u_stat} and the above inequality, we conclude that \eqref{eq:borne du/dx} holds true. 
        \\
       
\noindent {\em Step 2. The singularity $N(t) \delta (x)$.}	One of the major objections to regularity is the Dirac singularity in the right hand side of Eq.~\eqref{eq:FPSL lin}. To overcome it, we first introduce the solution to the  heat equation with a source
  \begin{equation}
     \left\{
    \begin{aligned}
        &\partial_t\theta = \partial_x^2 \theta +N(t)\delta_0,\\
     &\theta(t=0) = 0.
    \end{aligned}
     \right.
\end{equation}
            Since, as in step~1, we have $|N(t)| \leq C_0$,  we also have, for all $T>0$, 
    \begin{equation} \label{est:theta}
    \begin{cases}
           & \theta, \; \partial_x\theta \in L^\infty \big((0,T)\times \R \big), \qquad 
           \frac{ |\theta(x,t)|}{|x|} +|\partial_x\theta (x,t)| = o\big(\exp(-\frac{x^2}{4t})\big), \quad \frac{x^2}{t}\gg 1,
            \\[5pt]
           & \theta \in L^2\big((0,T);H^1(\R)\big).
    \end{cases}
    \end{equation}
    \
    Indeed, treating only the derivative, for $t\geq 0$, $x\in \R\setminus\{0\}$, we have
    \begin{equation} \label{eq:theta}
     \theta(t,x) = \int_0^t\frac{N(s)}{\sqrt{4\pi(t-s)}}\exp\big(-\frac{x^2}{4(t-s)}\big)ds, 
    \end{equation}
            \[
            \partial_x\theta(x,t) = -\frac{x}{4\sqrt{\pi}}\int_0^t\frac{N(s)}{(t-s)^{3/2}}\exp\big(-\frac{x^2}{4(t-s)}\big)ds,
            \]
            and we may write, with $\tau = t-s$, $z=\frac{x^2}{4\tau}$,
            \begin{align*}
                |\partial_x\theta(x,t)|
                &\leq \frac{C_0 |x|}{4\sqrt{\pi}} \int_0^t\tau^{-3/2}\exp\Big(-\frac{x^2}{4\tau}\Big)\,d\tau
                \leq \frac{C_0|x|}{4\sqrt{\pi}}\int_{\frac{x^2}{4t}}^\infty\Big(\frac{x^2}{4z}\Big)^{-3/2}\frac{x^2}{4z^2} e^{-z}dz
                \\
                &=\frac{C_0}{2\sqrt{\pi}} \int_{\frac{x^2}{4t}}^\infty z^{-1/2}e^{-z}\,dz \leq \frac{C_0}{2}.
            \end{align*}
            Hence, the $L^\infty$ estimate \eqref{est:theta} is proved. The $L^2$ estimate follows from this bound and the standard estimate for the heat equation.
 \\
 
\noindent {\em Step 3. We show that $u\in L^\infty \big((0,T);H^1_{loc}(\R) \big)$ when  $u_0\in \Cc^0_c(\R) \cap  H^1 (\R)$.}

            We may now remove the Dirac mass on the right hand side of Eq.~\eqref{eq:FPSL lin} and prove that 
            \begin{equation}\label{eq: q in L2}
                 q:= u- \theta \in L^2\big((0,T); H^1(\R)\big).
            \end{equation}
            Indeed, since $|u|\leq Cu_\infty \in L^2(\R)$, then $u\in L^2((0,T)\times \R)$. In addition, according to~\eqref{eq:borne du/dx}, 
            \[
            \int_0^T\int_\R\big(\partial_xu(t,x)\big)^2dx\,dt\leq \|u_\infty \|_\infty\int_0^T\int_\R\big(\partial_xu(t,x)\big)^2u_\infty^{-1}(x)\,dx\,dt<\infty.
            \]
            Hence, $u\in L^2\big((0,T); H^1(\R)\big)$ and thanks to~\eqref{est:theta}, we conclude.
            
            The function $q$ satisfies the equation
            \begin{equation}
            \left\{
            \begin{aligned}
                &\partial_tq = -\partial_x\big(hq\big)+\partial_x^2 q + \partial_x\big(h\theta\big), \\
                &q(t=0) = u_0.
            \end{aligned}
            \right.
            \end{equation}
            Differentiating this equation in $x$, we get
            \[
            \partial_t\big(\partial_xq\big)=-\partial_x^2\big(hq\big)+\partial_x^3q +\partial_x^2\big(h\theta\big).
            \]
            Let $\chi\in\testfunc(\R)$, $\chi\geq 0$ and define $I_\chi(t) = \frac{1}{2}\int_\R\chi(x)\big(\partial_xq(t,x)\big)^2dx$. Then, we compute
 \begin{align*}
                I'_\chi(t)
                &=\int_\R\chi\partial_xq\big( -\partial_x^2\big(hq\big)+\partial_x^3q +\partial_x^2(h\theta)\big)\,dx
                \\
                &=\int_\R \chi\partial_xq\big(-h''q-2h'\partial_xq -V\partial^2_xq\big)\,dx-\int_\R\big(\chi'\partial_xq +\chi\partial_x^2q\big)\big(\partial_x(h\theta) + \partial_x^2q\big)\,dx
                \\
                &=-\underbrace{\int_\R \Big[\chi h'' q(\partial_x q) -2 h'\chi\big(\partial_xq\big)^2
                -\frac 12 \chi h \partial_x (\partial_x q)^2 - \chi'(\partial_xq)\partial_x\big(h\theta\big) \big]\,dx}_{=I_1}
                \\
                &\quad -\underbrace{\int_\R\chi'(\partial_xq)(\partial_x^2q)\,dx}_{=I_2}
                -\underbrace{\int_\R\chi(\partial_x^2q)\partial_x(h\theta)\,dx}_{=I_3}-\underbrace{\int_\R\chi\big(\partial_x^2q\big)^2\,dx}_{=I_4} .
            \end{align*}
  Integrating by parts the third term of $I_1$ and using the Cauchy-Schwarz inequality for the fourth term, we obtain by immediate calculations for some constant $C$
   \[
   |I_1(t)| \leq C \int_\R [ q^2+(\partial_x q )^2]\,dx \big[ \|\chi h''\|_\infty + \|\chi h'\|_\infty +\|\chi' h\|_\infty+1 \big]+ \int_\R |\chi' \partial_x(h\theta)|^2.
   \]
   Thanks to the bounds \eqref{est:theta} and \eqref{eq: q in L2}, we conclude that for all $T>0$ there is a constant $C(T)$ such that
   \[
 \forall t\leq T, \qquad   |I_1(t)| \leq C(T).
   \]
   Next, the term $I_2$ can be estimated, for all $t\leq T$, as
   \[
   |I_2(t)|= \frac 12 |\int_\R \chi'\partial_x (\partial_x q)^2| =\frac 12 \int_\R \chi''(\partial_x q)^2 \,dx \leq C(T)
   \]
Finally, we use the Cauchy-Schwarz inequality to get, for all $t\leq T$,
\begin{align*}
 |I_3(t)| &\leq \frac 12 \int_\R \chi \big(\partial_x^2q\big)^2\,dx+ \frac 12 \int_\R \chi |\partial_x(h\theta)|^2\,dx
 \leq \frac 12 I_4(t) +C(t) 
\end{align*} 
These estimates show that for all $t\leq T$
 \begin{align*}
                I'_\chi(t) \leq C(T) - \frac 12 \int_\R \chi \big(\partial_x^2q\big)^2\,dx.
\end{align*}                 
Consequently we have $I_\chi(t)\leq C(T)+ I_\chi(0)$, and thus being given the regularity of $\theta$, we have obtained the desired regularity $u\in L^\infty \big((0,T);H^1_{loc}(\R) \big)$. 

Notice that we also deduce that  $q\in L^2 \big((0,T);H^2_{loc}(\R) \big)$.
\\

%endcommentout
            
            \noindent {\em Step 4. Local space and time continuity.} The regularity $u\in L^\infty(\R^+;H^1_{loc}(\R))$ ensures that for almost every $t\geq 0$, $u(t)$ is uniformly in $t$ locally H\"older continuous in $x$. Indeed, for all $t\leq T$, $|x|, \; |y| \leq R$
        \[
        |u(t,x)-u(t,y)| \leq \int_{z\in [x,y]} | \partial_x u (t,z)|\, dz \leq |x-y|^{\frac 12}
        \;\sup_{0\leq t \leq T} \int_{|x| \leq R} | \partial_x u (t,x)|^2\,dx.
        \]

            We turn to the H\"older estimate $O(h^{\frac 14})$ in time. Let $\phi\in \testfunc(\R)$ and $h>0$. We have, using the definition of weak solutions 
            \begin{align*}
                \int_\R\big(u(t+\eta)-u(t)\big)\phi\,dx &= \int_\R \int_t^{t+\eta} \partial_t u(\tau) \phi d\tau\,dx
                \\
                &=\int_t^{t+\eta}\bigg(\int_\R \big(h u(\tau)+\partial_xu(\tau)\big)\phi'\,dx+N_u(\tau)\phi(0)\bigg)\,d\tau.
            \end{align*}
          Then, we use the bounds $|u(t)|\leq C_0 u_\infty$, $|N(t)| \leq C_0 N_\infty$ and $\int_{-A}^A\big(\partial_xu(t)\big)^2dx \leq M_A$, to obtain a constant $C_A>0$ such that for all $ \phi\in \testfunc(\R)$ with supp$(\phi)\subset(-A,A)$,  we have for all $t\geq 0$, 
            \begin{align*}
                \bigg|\int_\R\big(u(t+\eta)-u(t)\big)\phi\,dx\bigg| \leq C_A\eta\big(\|\phi'\|_1+\|\phi'\|_2+\|\phi\|_\infty\big) = \alpha(A,\eta,\phi).
            \end{align*}
            Now, for $\varepsilon\in (0,1)$, we recall the regularization  $\ureg = u\ast \reg$ of Section~\ref{sec:unique} and write
            \begin{align*}
                |u(t+\eta,x)-u(t,x)|&\leq |u(t+\eta,x)-\ureg(t+\eta,x)|+|\ureg(t+\eta,x)-\ureg(t,x)|+|u(t,x)-\ureg(t,x)|.
            \end{align*}
            Then, for $x\in \R$, $T>0 $ and $ t\in (0,T)$, using the estimate \eqref{eq:borne du/dx}, we get
            \begin{align}
                |u(t,x)-\ureg(t,x)|&=\bigg|\int_{x-\eps}^{x+\eps}\big(u(t,x)-u(t,y)\big)\reg(x-y)\,dy\bigg|  \notag 
                \\
                & \leq \int_{x-\varepsilon}^{x+\varepsilon}\bigg|\int_x^y\partial_xu(t,z)dz\bigg|\,\reg(x-y)\,dy \notag 
                \\
                &\leq \int_{x-\eps}^{x+\eps}\bigg(\int_{(x,y)}
                \frac{\big(\partial_xu(t,z)\big)^2}{u_\infty(z)}dz \int_{(x,y)} u_\infty(z) \bigg)^{1/2} \reg(x-y)dy  \notag 
                \\
                &\leq \bigg(\sup_{t\in(0,T)}\int_{x-\eps}^{x+\eps}
                \frac{\big(\partial_xu(t,z)\big)^2}{u_\infty(z)} \, dz\bigg)^{1/2}\sqrt{\|u_\infty\|_\infty \varepsilon} \leq  C \sqrt{\varepsilon}. \label{est:errorConv}
            \end{align}
            The same estimate holds for $|\ureg(t+\eta,x)-u(t+\eta,x)|$. For the last term, we notice that for $\phi = \reg$, we have $\alpha(A,\eta,\reg) = \frac{C_A\eta}{\varepsilon^{3/2}}$.  We get, for any $\varepsilon\in (0,1)$, 
            \begin{align*}
                |u_\varepsilon(t+\eta,x)-u_\varepsilon(t,x)|&\leq 2C\sqrt{\varepsilon}+\alpha(A,\eta,\reg) \leq C'\big(\sqrt{\varepsilon} +\frac{\eta}{\varepsilon^{\frac 3 2}}\big).
            \end{align*}
            Therefore we have obtained the announced inequality 
            \begin{equation}\label{eq: u holder}
                \sup_{t\in (0,T)}\sup_{|x|\leq A}|u(t+\eta,x)-u(t,x)|\leq C'\inf_{\varepsilon\in (0,1)}\bigg\{\sqrt{\varepsilon}+\frac{\eta}{\varepsilon^{\frac 3 2}}\bigg\} =C(T,A) \eta^{1/4}.
            \end{equation}

        \noindent {\em Step 5. Global continuity.} On the one hand, for any $\varepsilon>0$, there is $A>0$ such that $\sup_{|x|>A}u_\infty\leq \frac{\varepsilon}{2C_0}$.
            %$\int_{|x|>A} u_\infty(x) \, dx \leq \frac{\eps}{C_0}$. 
            Therefore for any $\eta>0$, we have 
            \[
           \sup_{|x|\geq A}|u(t+\eta,x)-u(t,x)|\leq 2C_0 \sup_{|x|\geq A}u_\infty(x)\leq \varepsilon,
            \]
            %\[ \int_{|x|\geq A}|u(t+\eta,x)-u(t,x)|\,dx\leq 2C_0 \int_{|x|\geq A}u_\infty(x)\,dx\leq \varepsilon. \]
           On the other hand, fixing $A$ and using estimate \eqref{eq: u holder}, for $\eta \leq \frac{\eps}{C(T,A)}$, it holds for all $t\leq T$
            \[
            \sup_{|x|\leq A}  |u(t+\eta,x)-u(t,x)|\leq C(T,A)\eta \leq \eps.
            \]

            Altogether, we have obtained that 
            \(
            \sup_{x\in\R}|u(t+\eta,x)-u(t,x)|\leq \varepsilon.
        %    \quad \text{and}\quad \int_\R|u(t+\eta,x)-u(t,x)|\,dx\leq \varepsilon.
            \)
            This proves that 
            \begin{equation}\label{eq: conitnuité de u(t)}
                \lim_{\eta\rightarrow 0} \sup_{t\in (0,T)}\|u(t+\eta,\cdot) - u(t,\cdot)\|_{\infty} = 0.
                %\quad \text{and}\quad \lim_{h\rightarrow 0}\sup_{t\in (0,T)}\|u(t+\eta,\cdot) - u(t,\cdot)\|_{1} = 0.
            \end{equation}
            
            \noindent {\em Step 6.} We show that $u \in \Cc^0\big(\R^+;L^1(\R)\big)$. 

            Let $u_0\in L^1(\R)$ and $(u_0^{(n)})_{n\geq 0}\in \big(\Cc^0_c(\R)\big)^\N$ such that $\lim_{n\rightarrow\infty}\|u_0^{(n)}-u_0\|_1=0$. We denote $u^{(n)}(t)$ the solution toEq.~\eqref{eq:FPSL lin} associated to the data $u_0^{(n)}$. Then,  for any $h>0, t\geq 0$ and $n\in \N$, we have 
            \[
            \|u(t+\eta)-u(t)\|_1\leq \|u(t+\eta)-u^{(n)}(t+\eta)\|_1+\|u^{(n)}(t+\eta)-u^{(n)}(t)\|_1+\|u(t)-u^{(n)}(t)\|_1 .
            \]
            Now, using  the $L^1$ contraction property \eqref{eq:continuité S}, we find 
            \[
            \|u(t+\eta)-u(t)\|_1\leq2\|u_0-u_0^{(n)}\|_1+\|u^{(n)}(t+\eta)-u^{(n)}(t)\|_1, 
            \]
            and since $\lim_{\eta\rightarrow0}\|u^{(n)}(t+\eta)-u^{(n)}(t)\|_1=0$ according to \eqref{eq: conitnuité de u(t)}, we conclude
            \[
            \fa n\in \N, \qquad \limsup_{\eta\rightarrow0}\|u(t+\eta)-u(t)\|_1\leq 2\|u_0-u_0^{(n)}\|_1 \to 0 \quad \text{as} \quad n\to \infty,
            \]
           % Letting $n\rightarrow\infty$, we obtain
            %\[\lim_{\eta\rightarrow 0}\|u(t+\eta)-u(t)\|_1=0,\]
            and we have obtained that $t\mapsto u(t)\in \Cc^0(\R^+;L^1(\R))$.
         \end{proof}
         
\subsection{The strong form of the condition at infinity}

We now have the tools to prove that the boundary condition at infinity holds in a stronger form.

	\begin{prop}
	     Let $u_0\in \mathcal{C}^0_c(\R)$, $T>0$. Then, for $\phi \in \testfunc(\R^+)$ with supp$(\phi)\subset[0,T)$,
	     	\begin{equation}    \label{eq:SFCI}
	  \lim_{R\rightarrow\infty} \int_0^T\phi(t)[ h(R)u(t,R) - N_u(t)]\,dt = 0.
	  \end{equation}   
	   \end{prop}
	    	In other words, in a weak sense in $t$, 
	    \(
	    \lim_{R\rightarrow\infty}h(R)u(t,R) = N_u(t).
	    \)

	\begin{proof}
	    With the notation of the proposition, $R>x_1$ and $\varepsilon\in (0,1)$, we use again $\ureg(t,x) = u(t)\ast\reg(x)$. We decompose the integral as follows
	    \begin{align*}
	        \int_0^T&\phi(t)\big(h(R)u(t,R) - N_u(t)\big)\,dt =I_1+I_2+I_3
	        \\
	        &=\int_0^T\phi(t) h(R) \big(u(t,R)-\ureg(t,R)\big)\,dt+\int_0^T\phi(t)\Big(h(R)\ureg(t,R)-\int_\R h(x)\reg(x-R)u(t,x)\,dx \Big)\,dt
	        \\
	        & \quad +\int_0^T\phi(t)\Big(\int_\R h(x)\reg(x-R)u(t,x)\,dx -N_u(t)\Big)\,dt.
	    \end{align*}
	    
	    To estimate $I_1$, we use the method for \eqref{est:errorConv} with $x=R$ and get
	    \begin{align*}
            |u(t,R) - u_\varepsilon(t,R)|
            &\le 
                \left( \int_{R-\varepsilon}^{R+\varepsilon} 
                    \frac{(\partial_y u(t,y))^2}{u_\infty(y)}\,dy \right)^{1/2}
                \left( \int_{R-\varepsilon}^{R+\varepsilon} 
                    u_\infty(y)\,dy \right)^{1/2}
                \int_{R-\varepsilon}^{R+\varepsilon} 
                    \reg(x - R)\,dx\\
            &\le \left( \int_{R-\varepsilon}^{R+\varepsilon} 
                    \frac{(\partial_y u(t,y))^2}{u_\infty(y)}\,dy \right)^{1/2}
                \big(2\varepsilon u_\infty(R-\varepsilon)\big)^{1/2},
        \end{align*}
        where we used that, according to Proposition \ref{prop:u_stat}, $u_\infty$ is non-increasing on $(x_1,\infty)$.
        %and thus we have $\sup_{|x-R|\leq\varepsilon}u_\infty(x)=u_\infty(R-\varepsilon)$.
    
        Also, applying the Cauchy-Schwarz inequality, we get
        \begin{align*}
            \int_0^T|\phi(t)|\left(\int_{R-\varepsilon}^{R+\varepsilon}\frac{\big(\partial_x u(t,x)\big)^2}{u_\infty(x)}\,dx\right)^{1/2}\,dt
            & \leq \|\phi\|_2\underbrace{\left(\int_0^T\int_{R-1}^{R+1}\frac{\big(\partial_x u(t,x)\big)^2}{u_\infty(x)}\,dx\,dt\right)^{1/2}}_{=:\delta(R)}.
        \end{align*}
        These give, since $u_\infty(x)$ is decreasing for $x$ large enough, 
        \begin{align*}
            \left|I_1 \right|&\leq C_\phi h(R)\big(\varepsilon u_\infty(R-\varepsilon) \big)^{1/2}\delta(R).
        \end{align*}
        
	    Next, we address the second term. We have
        \begin{align*}
	        \Big|h(R)\ureg(t,R)&-\int_\R h(x)\reg(x-R)u(t,x)\,dx\Big| = \Big|\int_{R-\varepsilon}^{R+\varepsilon}\big(h(x)-h(R)\big)\reg(x-R)u(t,x)\,dx\Big|
    \\
    &\leq \sup_{y\in [R-\varepsilon, R+\varepsilon]} h'(y) \int_{R-\varepsilon}^{R+\varepsilon}|x-R|\reg(x-R)|u(t,x)|\,dx
    \\
            &\leq  C \varepsilon \sup_{y\in [R-\varepsilon, R+\varepsilon]} h'(y) u_\infty(R-\varepsilon),
        \end{align*}
where we used the monotonicity of \(u_\infty\) for $x$ large. Therefore, we obtain 
\[
I_2 \leq C_\phi \varepsilon \sup_{y\in [R-\varepsilon, R+\varepsilon]} h'(y) \;  u_\infty(R-\varepsilon).
\]

	    Finally, for the third term, we consider the test function $\rho_R^\varepsilon(x) = \int_x^\infty\reg(y-R)dy$ in  Eq.~\eqref{eq:egality mass conservation}. Assuming that $R>1$ and $\varepsilon\in (0,1)$, then, $\rho_R^\varepsilon(0)=1$. We obtain 
	    \begin{align}
            \int_0^T\bigg\{-\phi'(t)\int_\R u(t,x)\rho_R^\varepsilon(x)\,dx 
    - \phi(t) & \int_\R u(t,x)\big(h(x)(\rho_R^\varepsilon) '(x) + (\rho_R^\varepsilon)''(x)\big)\,dx\bigg\}\,dt  \notag
    \\
    &= \int_0^T N_u(t)\phi(t)\,dt + \phi(0)\int_\R u_0(x)\rho_R^\varepsilon(x)\,dx \label{est:SCI_I3}.
        \end{align}
        
        To treat this term, we make two observations. Firstly, the mass conservation \eqref{eq:Conservation mass} for $u$ gives
        \begin{align*}
            \int_0^T\phi'(t)\int_\R u(t,x)\rho_R^\varepsilon(x)\,dx +\phi(0)\int_\R u_0(x)\rho_R^\varepsilon(x)\,dx 
            & =\int_0^T\phi'(t)\int_\R u(t,x)(\rho_R^\varepsilon(x)-1)\,dx\\ 
            & +\phi(0)\int_\R u_0(x)(\rho_R^\varepsilon(x)-1)\,dx.
        \end{align*}
        Because $\rho_R^\varepsilon(x)-1 = 0$ for any $x\leq R-\varepsilon$ and  since for all $t\geq 0$, $|u(t)|\leq C_0 u_\infty$ we conclude
        \begin{align*}
            \left|\int_0^T\phi'(t)\int_\R u(t,x)\rho_R^\varepsilon(x)\,dx +\phi(0)\int_\R u_0(x)\rho_R^\varepsilon(x)\,dx \right|
            &\leq C_0 \big(\|\phi\|_\infty+\|\phi'\|_1\big)\underbrace{\int_{R-1}^\infty u_\infty(x)\,dx}_{=:\beta(R)} .
        \end{align*}
        Secondly, arguing as before for $u_\infty$, we have for any $t\geq 0$,
        \begin{align*}
            \left|\int_\R u(t,x)(\rho_R^\varepsilon)''(x)\,dx\right| &
            = \left|\int_\R\partial_x u(t,x)(\rho_R^\varepsilon)'(x)\,dx\right|  = \left|\int_{R-\varepsilon}^{R+\varepsilon}\partial_x u(t,x)\reg(x-R)\,dx\right|
            %\\& \leq \left(\int_{R-\varepsilon}^{R+\varepsilon}\frac{\big(\partial_x u(t,x)\big)^2}{u_\infty(x)}\,dx\right)^{1/2}\left(\int_{R-\varepsilon}^{R+\varepsilon}\reg(x-R)^2 u_\infty(x)\,dx\right)^{1/2}
            %\\& \leq \left(\int_{R-\varepsilon}^{R+\varepsilon}\frac{\big(\partial_x u(t,x)\big)^2}{u_\infty(x)}\,dx\right)^{1/2}\big(\sup_{|x-R|\leq\varepsilon}u_\infty(x)\big)^{1/2}\underbrace{\left(\int_{R-\varepsilon}^{R+\varepsilon}\reg(x-R)^2dx\right)^{1/2}}_{=\frac{K}{\sqrt{\varepsilon}}}
            \\
             & \leq \left(\int_{R-\varepsilon}^{R+\varepsilon}\frac{\big(\partial_x u(t,x)\big)^2}{u_\infty(x)}\,dx\right)^{1/2}u_\infty(R-\varepsilon)^{1/2}\underbrace{\left(\int_{R-\varepsilon}^{R+\varepsilon}\reg(x-R)^2dx\right)^{1/2}}_{=\frac{C}{\sqrt{\varepsilon}}}.
        \end{align*}
        
        So we get for any $T>0$, and arguing as in the first term,
        \begin{align*}
            \left|\int_0^T \phi(t)
                \int_\R u(t,x)(\rho_R^\varepsilon)''(x)\,dx\,dt\right|
                &\leq 
                C \left(\frac{u_\infty(R-\varepsilon)}{\varepsilon}\right)^{1/2}
                \int_0^T |\phi(t)|
                \left(\int_{R-\varepsilon}^{R+\varepsilon}
                    \frac{(\partial_x u(t,x))^2}{u_\infty(x)}\,dx\right)^{1/2} dt
                \\
            &\leq C_\phi \left(\frac{u_\infty(R-\varepsilon)}{\varepsilon}\right)^{1/2}  \delta(R).
        \end{align*}
        
        Using the two estimates, we obtain from \eqref{est:SCI_I3}
        \begin{align*}
          |I_3|& =  \left|\int_0^T\phi(t)\big(\int_\R u(t,x)h(x)(\rho_R^\varepsilon)'(x)\,dx -N_u(t)\big)\,dt\right| 
          \\
          & = \bigg|\int_0^T\phi'(t)\int_\R u(t,x)\rho_R^\varepsilon(x)\,dx +\phi(0)\int_\R u_0(x)\rho_R^\varepsilon(x)\,dx +\int_\R u(t,x)(\rho_R^\varepsilon)''(x)\,dx\bigg|
        \\
            &\leq C_\phi\bigg(\beta(R)+\bigg(\frac{u_\infty(R-\varepsilon)}{\varepsilon}\bigg)^{1/2}\delta(R)\bigg).
        \end{align*}
        By gathering the three terms $I_1, \; I_2,\; I_3$, we finally obtain,
\begin{align}
    \Big|\int_0^T& \phi(t)\big(h(R)u(t,R) - N_u(t)\big)\,dt\Big| \notag
%    &\leq \tilde{C}_\phi\Big(h(R)
%        (\varepsilon u_\infty(R-\varepsilon)))^{1/2}\delta(R,\varepsilon)
%        + \varepsilon h'(R+\varepsilon)u_\infty(R-\varepsilon)
%        \\  &+ \beta(R)  + \big(\frac{u_\infty(R-\varepsilon)}{\varepsilon}\big)^{1/2}\delta(R,\varepsilon) \Big)
\\
    & \leq {C}_\phi\Big(
        u_\infty(R-\varepsilon)^{1/2}\delta(R)\big(h(R)\varepsilon^{1/2} + \varepsilon^{-1/2} \big)
        + \varepsilon \sup_{y\in [R-\varepsilon,R+\varepsilon]}h'(y) \; u_\infty(R-\varepsilon) + \beta(R) \Big).  \label{SCI:total}
\end{align}
% NEW VERSION
\\
The difficulty in estimating these terms is that  we have to finely tune  the ways  $\varepsilon \to 0$ and $R\to \infty$. To do that we have first to recall that, according to \eqref{eq:limh'/h^2},
\[
\omega(R):= \sup_{y\geq R}\frac{h'(y)}{h(y)^2} \to 0, \qquad \text{as} \quad R \to \infty.
\]
Now, we choose uniquely $\varepsilon= \varepsilon_R$ by the relation (notice that $\varepsilon \mapsto \frac{1}{h(R+\varepsilon)}$ is decreasing)
\[
\varepsilon_R = \frac{1}{h(R+\varepsilon_R)} \leq \frac{1}{h(R)}.
\]
Consequently, we also have, for $R$ large enough
\begin{equation} \label{est:epsR}
\varepsilon_R^{-1} u_\infty(R-\varepsilon_R) \leq 4 N_\infty.
\end{equation}
Indeed, we compute
\begin{align*}
 h(R+\varepsilon_R) u_\infty(R-\varepsilon_R)
&\leq  u_\infty(R-\varepsilon_R) \big[ h(R-\varepsilon_R)+2 \varepsilon_R  \sup_{y\in[R-\varepsilon_R, R+\varepsilon_R]} h'(y)\big] 
\\
&\leq  u_\infty(R-\varepsilon_R)  \big[ h(R-\varepsilon_R)+2 \varepsilon_R \omega(R-\varepsilon_R) h(R+\varepsilon_R)^2 \big]
\\
&=  u_\infty(R-\varepsilon_R) h(R-\varepsilon_R)+2  \omega(R-\varepsilon_R) h(R+\varepsilon_R) u_\infty(R-\varepsilon_R) \big.
\end{align*}
Consequently, choosing $\omega(R-\varepsilon_R) \leq \frac 14 $, since $u_\infty(x) h(x) \to N_\infty$ as $x \to \infty$, we have
\[
\frac 1 2 \varepsilon_R^{-1} u_\infty(R-\varepsilon_R) \leq h(R+\varepsilon_R) u_\infty(R-\varepsilon_R) (1-2  \omega(R-\varepsilon_R)) \leq u_\infty(R-\varepsilon_R) h(R-\varepsilon_R) \leq 2 N_\infty,
\]
still for $R$ large enough, and \eqref{est:epsR} is proved.
\\

We may now show that, with this choice of $\varepsilon_R$ the right hand side of \eqref {SCI:total} converges to $0$ as $R \to \infty$.

The first term is estimated as 
\[
u_\infty(R-\varepsilon_R)^{1/2}\delta(R)h(R)\varepsilon^{1/2}
\leq u_\infty(R-\varepsilon_R)^{1/2}\varepsilon^{-1/2}\delta(R) \leq 4 N_\infty \delta(R) \to 0.
\]

The second term is 
\begin{align*}
\varepsilon^{-1/2} u_\infty(R-\varepsilon_R)^{1/2}\delta(R)\leq 4 N_\infty \delta(R) \to 0.
\end{align*}

The next term we consider is, arguing as above, 
\begin{align*}
\varepsilon_R \sup_{y\in [R-\varepsilon_R,R+\varepsilon_R]}h'(y) \; u_\infty(R-\varepsilon_R)
&\leq \varepsilon_R  \omega(R-\varepsilon_R)  h(R+\varepsilon_R )^2 u_\infty(R-\varepsilon_R ) \leq  4 N_\infty \omega(R-\varepsilon_R).
\end{align*}

And finally, $\beta(R)$ converges to $0$ because $u_\infty$ is integrable.
\\

Therefore, we have proved the limit in \eqref{eq:SFCI}.
	\end{proof}

\subsection{Relative entropy relation}

As usual for all linear equations preserving positivity, \cite{RefMMP}, Eq.~\eqref{eq:FPSL lin} comes with a family of relative entropies defined for	$H \in \Cc^1(\R; \R)$ by
\begin{align}\label{def:entropy}
\Ee_H(t) := \int_\R H(w(t,x))u_\infty(x)\,dx, \qquad w(t) = \frac{u(t)}{u_\infty}.
\end{align}
The statements depend on the level of regularity of the solution.

\begin{theorem} [Relative entropy equality] \label{th:entropy}
		 Let $H \in \Cc^1(\R)$ be a convex function and $u_0$ an initial data such that $H(\frac{u_0}{u_\infty})u_\infty \in L^1(\R)$. Then  the entropy function  satisfies in the weak sense
		\[
		\fa t\geq 0, \qquad \frac{d}{dt}\Ee_H(t)\leq 0.
		\]
		Moreover, if $|u_0|\leq C_0 u_\infty$ as in Proposition~\ref{prop:MaxPrinc}, then, the entropy satisfies the equality 
		\begin{align}\label{ineq:entropy}
			\fa t\geq 0, \qquad \frac{d}{dt} &\Ee_H(t)= JD_H(t)
			-\int_\R\big(\partial_x w(t,x)\big)^2H''(w(t,x))u_\infty(x)\,dx\leq 0
		\end{align}
where the jump  dissipation $JD_H(t)$ is defined, for any $\chi \in \testfunc(\R)$ such that $\int_\R\chi(x)\,dx = 1$, by
\[
\frac{JD_H(t)}{N_\infty} =- \lim_{R\rightarrow\infty}\int_\R\chi(x-R)H\Big(w(t,x)\Big)\,dx-H\Big(w(t,0)\Big)-H'\Big(w(t,0)\Big)\Big(\frac{N_u(t)}{N_\infty}-w(t,0)\Big)  \leq 0.
\]
	\end{theorem}

\begin{remark}
    \begin{enumerate}
    \item Notice that, according to \eqref{BoundaryCond}, we have 
    \[
    \lim_{R\rightarrow \infty}\int_\R \chi(x-R)H\Big(w(t,x)\Big)dx\geq H\Big(\frac{N(t)}{N_\infty}\Big).
    \]
    Therefore, since $H$ is convex, $JD_H$ satisfies
    \[
    \frac{JD_H(t)}{N_\infty}\leq -\bigg(H\Big(\frac{N(t)}{N_\infty}\Big)-H\Big(w(t,0)\Big)-H'\Big(w(t,0)\Big)\Big(\frac{N_u(t)}{N_\infty}-w(t,0)\Big)\bigg)\leq 0.
    \]
        \item We give a proof through based on direct computations on the solution to \eqref{eq:FPSL lin}. A more direct, yet equally complex, approach involves passing to the limit in the entropy equality associated with the truncated problem \eqref{eq:Entropie trunc}. The moment estimate for the solution obtained via equality \eqref{eq: U_T} allows for the control of terms as $R \rightarrow \infty$.
    \end{enumerate}
\end{remark}

	\begin{proof}
	Let $u_0 \in \mathcal{C}^1_c(\mathbb{R})$. By Proposition \ref{prop:regularity of u}, $u$ is continuous on $\mathbb{R}^+ \times \mathbb{R}$, which implies that the trace $t \mapsto w(t,0)$ is well-defined. Notice in addition that according to Proposition \ref{prop:MaxPrinc}, $w$ satisfies $w\in L^\infty(\R^+\times \R)$.
    Using Eq. \eqref{eq:FPSL lin}, $w$ satisfies the following equation
    \begin{align*}
        \partial_t w &= \frac{1}{u_\infty} \left( \partial^2_x(w u_\infty) - \partial_x(h w u_\infty) + N_u \delta_0 \right) \\
 %       &= \frac{1}{u_\infty} \left( w u''_\infty + 2 u'_\infty \partial_x w + u_\infty \partial_x^2 w - \partial_x(h u_\infty) w - h u_\infty \partial_x w + N_u \delta_0 \right) \\
        &= \partial_x^2 w + \partial_x w \left( 2 \frac{u'_\infty}{u_\infty} - h \right) + \frac{w}{u_\infty} \underbrace{(u''_\infty - \partial_x(h u_\infty))}_{= -N_\infty \delta_0} + \frac{N_u}{u_\infty} \delta_0 \\
        &= \partial_x^2 w + \partial_x w \left( 2 \frac{u'_\infty}{u_\infty} - h \right) + \left( \frac{N_u}{u_\infty} - N_\infty \frac{w}{u_\infty} \right) \delta_0.
    \end{align*}
    Let $R > 0$ and let $\chi$ be defined as in Theorem \ref{th:entropy}. We introduce the cut-off function $\rho_R(x) = \int_x^\infty \chi(y-R) \, dy$.
    Define the truncated entropy functional by
    \[
        I_R(t) = \int_{\mathbb{R}} \rho_R(x) H(w(t,x)) u_\infty(x) \, dx.
    \]
    Differentiating with respect to time for $t \geq 0$, we obtain:
    \begin{align*}
        I'_R(t) &= \int_{\mathbb{R}} \rho_R(x) \left\{ \partial_x^2 w + \partial_x w \left( 2 \frac{u'_\infty}{u_\infty} - h \right) + \left( \frac{N_u}{u_\infty} - N_\infty \frac{w}{u_\infty} \right) \delta_0 \right\} H'(w) u_\infty \, dx.
    \end{align*}
    Handling the Dirac term at $x=0$,  the expression becomes
    \begin{equation} \label{eq:IR_deriv_interm}
        \begin{aligned}
            I'_R(t) &= - \rho_R(0) N_\infty H'(w(t,0)) \left( w(t,0) - \frac{N_u(t)}{N_\infty} \right) \\
            &\quad + \int_{\mathbb{R}} \rho_R(x) \left\{ \partial_x^2 w u_\infty + \partial_x w (2 u'_\infty - h u_\infty) \right\} H'(w) \, dx.
        \end{aligned}
    \end{equation}
  Next we use the identity $H'(w) \partial_x^2 w = \partial_x^2 H(w) - (\partial_x w)^2 H''(w)$ and integrating by parts, the term involving $\partial_x^2 H(w)$ becomes
    \begin{align*}
        \int_{\mathbb{R}} \rho_R u_\infty \partial_x^2 H(w) \, dx
%        &= - \int_{\mathbb{R}} \partial_x(\rho_R u_\infty) \partial_x H(w) \, dx \\
        &= - \int_{\mathbb{R}} \rho'_R u_\infty \partial_x w H'(w) \, dx - \int_{\mathbb{R}} \rho_R u'_\infty \partial_x w H'(w) \, dx.
    \end{align*}
    Substituting this back into \eqref{eq:IR_deriv_interm}, we get
    \begin{align*}
        I'_R(t) &= - \rho_R(0) N_\infty H'(w(t,0)) \left( w(t,0) - \frac{N_u(t)}{N_\infty} \right) - \int_{\mathbb{R}} \rho'_R u_\infty \partial_x w H'(w) \, dx\\
        &\quad  - \int_{\mathbb{R}} \rho_R (\partial_x w)^2 H''(w) u_\infty \, dx + \int_{\mathbb{R}} \rho_R \partial_x w H'(w) \underbrace{(u'_\infty - h u_\infty)}_{= -N_\infty \mathbf{1}_{x \geq 0}} \, dx.
    \end{align*}
    Let us analyze the last integral involving the indicator function. We compute
    \begin{align*}
        \int_{\mathbb{R}} \rho_R \partial_x w H'(w) (-N_\infty \mathbf{1}_{x \geq 0}) \, dx
        &= - N_\infty \int_0^\infty \rho_R(x) \partial_x H(w(t,x)) \, dx \\
        &= - N_\infty \left[ \rho_R(x) H(w(t,x)) \right]_{0}^\infty + N_\infty \int_0^\infty \rho'_R(x) H(w(t,x)) \, dx \\
        &= N_\infty \rho_R(0) H(w(t,0)) + N_\infty \int_{\mathbb{R}} \rho'_R(x) H(w(t,x)) \, dx.
    \end{align*}
    Now we pass to the limit as $R \to \infty$.
    Using the bound $u_\infty^{1/2} \partial_x w \in L^2(\mathbb{R}^+ \times \mathbb{R})$ from \eqref{eq: borne d_x w } and that $H'(w) \in L^\infty$, the Cauchy-Schwarz inequality gives
    \[
        \left| \int_{\mathbb{R}} \rho'_R u_\infty \partial_x w H'(w) \, dx \right| \leq C \left( \int_{\mathbb{R}} |\rho'_R| u_\infty \, dx \right)^{1/2} \left( \int_{\mathbb{R}} |\rho'_R| u_\infty (\partial_x w)^2 \, dx \right)^{1/2}.
    \]
    Thus, for almost every $t \geq 0$, this term vanishes as $R \to \infty$.
    Furthermore, $\lim_{R \to \infty} \rho_R(0) = 1$, and by the definition of $\rho_R$ we have
    \[
        \lim_{R \to \infty} N_\infty \int_{\mathbb{R}} \rho'_R(x) H(w(t,x)) \, dx = \lim_{R \to \infty} - N_\infty \int_{\mathbb{R}} \chi(x-R) H(w(t,x)) \, dx.
    \]
    Recalling that in the weak sense $\lim_{R \to \infty} \int_{\mathbb{R}} \chi(x-R) w(t,x) \, dx = \frac{N_u(t)}{N_\infty}$, and using Jensen's inequality, we have
    \[
        \liminf_{R \to \infty} \int_{\mathbb{R}} \chi(x-R) H(w(t,x)) \, dx \geq H \left( \frac{N_u(t)}{N_\infty} \right).
    \]
    Finally, gathering all terms and taking the limit $R \to \infty$, we obtain
    \begin{align*}
        \frac{d}{dt} \int_{\mathbb{R}} H(w) u_\infty \, dx
        &\leq - N_\infty \left[ H \left( \frac{N_u}{N_\infty} \right) - H(w(t,0)) - H'(w(t,0)) \left( \frac{N_u}{N_\infty} - w(t,0) \right) \right]
        \\
        &\quad - \int_{\mathbb{R}} (\partial_x w)^2 H''(w) u_\infty \, dx.
    \end{align*}
    Since $H$ is convex, the term in bracket is non-negative, as well as the integral term. \\
    
    Now, we extend the result to general initial data. Assume $u_0 \in L^1(\mathbb{R})$ such that $H(\frac{u_0}{u_\infty})u_\infty \in L^1(\mathbb{R})$.
    Let $(u_0^{(n)})_{n \in \mathbb{N}}$ be a sequence in $\mathcal{C}^1_c(\mathbb{R})$ such that $\|u_0^{(n)} - u_0\|_{1} \to 0$ as $n \to \infty$.
    Let $w^{(n)}(t,x) = \frac{u^{(n)}(t,x)}{u_\infty(x)}$ be the corresponding solution.
    
    We first assume that $H$ is convex and satisfies $H' \in L^\infty(\mathbb{R})$.
    According to the previous inequality  \eqref{ineq:entropy}, for any non-negative test function $\phi \in \mathcal{C}_c^\infty(\mathbb{R}^+)$ and any $n \in \mathbb{N}$, we have:
    \[
        \int_0^\infty (-\phi'(t)) \int_{\mathbb{R}} H(w^{(n)}(t,x)) u_\infty(x) \, dx \, dt
        \leq \phi(0) \int_{\mathbb{R}} H \left( \frac{u_0^{(n)}(x)}{u_\infty(x)} \right) u_\infty(x) \, dx.
    \]
    We now verify the convergence of both sides. For the left-hand side, for any $n \in \mathbb{N}^*$, we estimate:
    \begin{align*}
        \bigg| \int_0^\infty (-\phi'(t)) &\int_{\mathbb{R}} \big( H(w^{(n)}) - H(w) \big) u_\infty \, dx \, dt \bigg| \\
        &\leq \int_0^\infty |\phi'(t)| \int_{\mathbb{R}} \|H'\|_\infty |w^{(n)}(t,x) - w(t,x)| u_\infty(x) \, dx \, dt \\
        &= \|H'\|_\infty \int_0^\infty |\phi'(t)| \, \|u^{(n)}(t) - u(t)\|_{1} \, dt \\
        &\leq \|H'\|_\infty \|\phi'\|_{L^1(\mathbb{R}^+)} \|u_0^{(n)} - u_0\|_{1} \xrightarrow[n \to \infty]{} 0,
    \end{align*}
    where we used $L^1$-contraction. 
    
    Similarly, for the initial data term, using the Lipschitz continuity of $H$:
    \[
        \lim_{n \to \infty} \int_{\mathbb{R}} H \left( \frac{u_0^{(n)}(x)}{u_\infty(x)} \right) u_\infty(x) \, dx = \int_{\mathbb{R}} H \left( \frac{u_0(x)}{u_\infty(x)} \right) u_\infty(x) \, dx.
    \]
    Passing to the limit as $n \to \infty$ in the entropy inequality, we obtain for every non-negative $\phi \in \mathcal{C}_c^\infty(\mathbb{R}^+)$:
    \[
        \int_0^\infty (-\phi'(t)) \int_{\mathbb{R}} H(w(t,x)) u_\infty(x) \, dx \, dt
        \leq \phi(0) \int_{\mathbb{R}} H \left( \frac{u_0(x)}{u_\infty(x)} \right) u_\infty(x) \, dx.
    \]
    
    Finally, to remove the assumption that $H'$ is bounded, consider a general continuous convex function~$H$.
    Let $(H_k)_{k \geq 0}$ be an increasing sequence of Lipschitz continuous convex functions approximating $H$ (e.g., via regularization and truncation) such that $H_k \to H$ pointwise as $k \to \infty$ and pass to the limit in the relative entropy inequality.
    \end{proof}
	
	\subsection{A Poincar\'e inequality}
	
	In the spirit of \cite{CCP} (Proposition 4.2 and Lemma 4.3), having this entropy inequality at hand, it is natural to ask whether it yields exponential convergence toward the stationary state in $L^2(u_\infty^{-1})$. 
It is well known that such a property is equivalent to a Poincaré inequality. 
According to \cite{Muckenhoupt1972,BartheR2008}, the latter holds provided that $h$ satisfies the growth condition \eqref{eq: borne V} at $+\infty$. 
	\begin{prop} [Poincar\'e inequality]
	Assume that the drift $h$ satisfies the additional property
    \begin{equation} \label{eq: borne V}
       \sup_{x\geq x_1}\bigg\{h(x)\int_x^\infty\frac{1}{h(y)}dy\bigg\} <\infty.
    \end{equation}
Then, there is a constant $\beta>0$ such that for any $f\in H^1(\R)$, we have the Poincar\'e inequality 
    \begin{equation} \label{eq:poincare}
        \int_\R\bigg(f(x)-\int_\R f(y)u_\infty(y)\,dy\bigg)^2u_\infty(x)\,dx\leq \beta\int_\R\big(f'(x)\big)^2u_\infty(x)\,dx.
   \end{equation}
Consequently, with $\mu=\beta^{-1}$, for any $u_0\in L^2(u_\infty^{-1})$ that satisfies $\int_\R u_0(x)dx = 1$, we have
	    \[
	    \forall t \geq 0, \qquad\|u(t)-u_\infty\|_{L^2(u_\infty^{-1})}\leq e^{-\mu t}\|u_0-u_\infty\|_{L^2(u_\infty^{-1})}.
	    \]
	\end{prop}
	\begin{proof}
	    According to estimate \eqref{eq:limh'/h^2}, for any $x\geq x_1$, we have
        \[
        U(x) :=\int_x^\infty u_\infty(y)dy\leq N_\infty\int_x^\infty\frac{dy}{h(y)}\leq \frac{M}{h(x)} \qquad\text{thanks to \eqref{eq: borne V}}.
        \]
        In addition, according to  \eqref{est:uinfty} and \eqref{eq:limh'/h^2}, we can find $x_2\geq x_1$ such that for $x\geq x_2$, $\sup_{y\geq x}\frac{h'(y)}{h(y)^2}\leq \frac{1}{2}$ and 
        \[
        u_\infty(x)\geq \frac{N_\infty}{2h(x)}\geq \frac{N_\infty}{2M}U(x).
        \]
        Now, we use that $u_\infty$ is Gaussian on $(-\infty,0)$ according to \eqref{eq:  u stat +infini}, so we have for any $x\leq -|h_0|+1$,
        \begin{align*}
            1-U(x) &=\int_{-\infty}^x u_\infty(y)dy=C_\infty\int_{-\infty}^xe^{-\frac{y^2}{2}+h_0y}dy\\
            &\leq C_\infty \int_{-\infty}^x\frac{h_0-y}{h_0-x}e^{-\frac{y^2}{2}+h_0y}dy \leq C_\infty \frac{1}{h_0-x}e^{-\frac{x^2}{2}+h_0x}\\
            &\leq u_\infty(x).
        \end{align*}
       % And finally, we get \[ \fa x\leq -|h_0|+1,\quad  1-U(x)\leq u_\infty(x)\]
        Therefore, we can find a constant $k>0$ such that for any $x\in \R$, we have
        \[
        (1-U(x))U(x)\leq ku_\infty(x).
        \]
        This property ensures that the Poincaré inequality with the weight $u_\infty$, i.e.,  \eqref{eq:poincare} holds true. We refer to \cite{Muckenhoupt1972} for more details on this property.

        Consequently, when $|u_0|\leq C_0 u_\infty$, the entropy inequality for $H(z)=(z-1)^2$ gives
        \begin{align*}
            \frac{d}{dt}\int_\R \big(w(t,x)-1\big)^2u_\infty(x)\,dx 
            &\leq -2\int_\R\big(\partial_xw(t,x)\big)^2u_\infty(x)\,dx\\
            &\leq -2\beta^{-1}\int_\R \bigg(w(t,x)-\int_\R w(t,y)u_\infty(y)\,dy\bigg)^2u_\infty(x)\,dx
        \end{align*}
        Using that $\int_\R w(t,x)u_\infty(x)\,dx = \int_\R u(t,x)\,dx = 1$, we finally get with 
        \[
         \frac{d}{dt}\int_\R \big(w(t,x)-1\big)^2u_\infty(x)\,dx\leq - 2\beta^{-1} \int_\R \big(w(t,x)-1\big)^2u_\infty(x)\,dx.
        \]
        We conclude with Gronwall Lemma. 
	\end{proof}

	\section{Long time convergence using Doeblin-Harris method}
	\label{sec:LongTermDH}
	
	Our purpose is to prove the convergence of the solution to\eqref{eq:FPSL lin} towards its stationary state as stated in Theorem~\ref{th:conv DH}. We adapt to our situation the Doeblin-Harris method, see \cite{Harris_1956, HairerM-Doeblin} and also for degenerate parabolic equations \cite{SaSm2024}. We begin with some notation and then detail the proof in the following subsections.
	\\

A direct and standard consequence of uniqueness in Theorem~\ref{th:existence weak sol} is that the Fokker-Planck Eq. \eqref{eq:FPSL lin} generates a contraction  semi-group $\big(S(t)\big)_{t\geq 0}$,  $S(t):L^1(\R)\rightarrow L^1(\R)$  and it is mass and positivity preserving linear operator, i.e.,  for all $t\geq 0$, $u_0\in L^1(\R)$
	    \begin{equation}\label{eq:continuité S}
	    \|S(t)u_0\|_1\leq \|u_0\|_1, \qquad \int_{\R} S(t)u_0(x) \, dx = \int_{\R} u_0(x) \, dx, \qquad u_0 \geq 0 \Rightarrow S(t)u_0 \geq 0.
	    \end{equation}
	    
The core of the Doeblin method relies on finding a constant $\alpha > 0$, a probability density $\nu$, and a time $t_0 \geq 0$ such that for any initial probability density $u_0 \in L^1(\R)$, the following lower bound holds:
\begin{equation} \label{eq: condition Doeblin}
    S(t_0)f \geq \alpha \nu, \qquad \qquad \qquad \text{(Doeblin condition)}.
\end{equation}
However, a major difficulty arises from the fact that the drift~$h$ is linear on $(-\infty, x_0)$. Consequently, if the support of the initial data $u_0$ is localized in $(-\infty, -A)$ for any $A>0$, the bulk of its mass at time~$t$ will remain concentrated in the interval $(-\infty, -A e^{-t})$. It is therefore impossible to find a uniform time $t_0$ with respect to all initial conditions to satisfy condition~\eqref{eq: condition Doeblin}.

To overcome this obstacle, we restrict our analysis to the weighted space $E$ defined as
\[
    E = \left\{ f \in L^1(\R) : \int_{\R} \big(1 + (-x)_+\big) |f(x)| \, dx < \infty \right\},
    \qquad \|f\|_E = \|(1 + (-x)_+) f\|_{L^1}.
\] 
 This space ensures that any family of functions uniformly bounded in $E$ exhibits mass tightness, specifically preventing mass escape towards $-\infty$. 
Moreover, the superlinear drift on $(0,\infty)$ sends mass to $+\infty$ in finite time. This yields  a uniform lower bound on the source term $\delta_0 \int_0^t N(t)\,dt$, which, by comparison with a suitably chosen subsolution, allows us to obtain the Doeblin condition \eqref{eq: condition Doeblin}.
\\

To prove Theorem \ref{th:conv DH}, we proceed in two steps.
The first step consists of proving the existence of a uniform minorization through the construction of a function $\nu$, as in equation~\eqref{eq: condition Doeblin},  under the assumption that the initial data are uniformly bounded in $E$ by a constant $C>0$.
To this end, we first show that the source term
\[
\delta_0 \int_0^t N(\tau)\, \mathrm{d}\tau
\]
admits a uniform lower bound, provided that the initial data satisfy the uniform bound $\|u_0\|_E \leq C$.
Next we construct a subsolution to the problem Eq.~\eqref{eq:FPSL lin}, via a solution of the heat equation with the above source term (and thus inheriting the corresponding lower bound).
In the second step, we establish two key estimates on the $E$-norm of the solution (see Lemma~\ref{lem:estimation u_E}), in order to apply the classical Doeblin--Harris arguments and introduce a suitable norm, defined as a combination of the $\|\cdot\|_{1}$ and $\|\cdot\|_{E}$ norms to identify a time $T>0$ such that the semigroup map $S(T)$ is a contraction.

\subsection{A uniform Lower Estimate }
For the following analysis, we write $S(t)u_0 = u(t)$.
As explained, we first show a uniform lower bound on the source term $\int_0^t N_u(\tau)d\tau$.
\begin{prop}\label{prop:Minoration source}
    There is a constant $c_\varphi>0$ such that for any $C> 0$, we can find a time $t_C\geq 0$ (that only depends on $C$) (i.e., $u_0\geq 0$ and $\int_\R u_0(x)\,dx=1$) such that for any initial probability density $u_0\in E$ with $\|u_0\|_{E}\leq C$, we have
    \begin{equation}\label{eq: minoration int N_u}
        \fa t \geq t_C, \qquad \int_0^t N_u(\tau)\,d\tau \geq \frac{1}{c_\varphi}.
    \end{equation}
\end{prop}

\begin{proof}
    To get the desired lower bound, we first introduce the function $\varphi$. For $x \in \mathbb{R}$, define
    \[
    \varphi(x)
    = \int_{0}^{x} e^{-\mathcal{H}(y)}
       \left( \int_{-\infty}^{y} e^{\mathcal{H}(z)}\,dz \right) dy .
    \] 
    This function satisfies the following properties
     \begin{equation}\label{eq:eqdif sur phi}
        \fa x \in \R, \qquad\varphi''(x)+h(x)\varphi'(x) = 1,
    \end{equation}
    \begin{equation}\label{eq:behavior phi}
        \varphi(x) \sim_{x\rightarrow -\infty}- \log|x|\:\:,\:\:
        \lim_{x\rightarrow \infty} \varphi(x) = c_\varphi <\infty \text{ and } \varphi \text{ is non-decreasing on }\R, 
    \end{equation}
    \begin{equation}\label{eq:signe phi}
        \varphi(x) \geq 0\Longleftrightarrow x\geq 0.
    \end{equation}
Also, since $|\varphi(x)| \leq M \big((-x)_++1\big)$, for a  constant $M\geq 0$, for any $f\in E$, we have
    \begin{equation}\label{eq:encadrement phi E}
        \int_\R|\varphi(x)f(x)|dx\leq M\int_\R\big((-x)_++1\big)|f(x)|dx=M\|f\|_E.
    \end{equation}
    Now, let $I(t) = \int_\R\varphi(x)u(t,x)\,dx$. By \eqref{eq:encadrement phi E} together with point~2 of Lemma~\ref{lem:estimation u_E}, this quantity is well defined for all $t \geq 0$.  Then, differentiating $I(t)$, we obtain
    \begin{align*}
        I'(t) &= \int_\R\big(-\partial_x\big(h(x)u\big)+\partial_x^2u + N_u(t)\delta_0\big)\varphi(x)\,dx \\
        &= \int_\R u(t,x)\big(h(x)\varphi'(x)+\varphi''(x)\big)\,dx - \varphi(0)N_u(t) \\
        &= \int_\R u(t,x)\,dx - c_\varphi N_u(t) =1-c_\varphi N_u(t).
    \end{align*}
    Therefore, integrating on $(0,t)$ we obtain thanks to  \eqref{eq:signe phi}
    $$
    \fa t\geq 0,\qquad \int_0^t N_u(\tau)\,d\tau = \frac{1}{c_\varphi}\big(t+I(0)-I(t)\big).
    $$
    Now, using the properties  \eqref{eq:encadrement phi E} and \eqref{eq:behavior phi} for $I(t)$, we estimate 
    \begin{align*}
        I(0)-I(t) &= \int_\R\varphi(x)\big(u_0(x)-u(t,x)\big)\,dx 
        \\
        &\geq -\int_{-\infty}^0|\varphi(x)|u_0(x)\,dx - \int_{0}^\infty|\varphi(x)|u(t,x)\,dx \\
        &\geq -M \|u_0\|_{E} - c_\varphi \int_0^\infty u(t,x)\,dx \quad \text{according to  \eqref{eq:encadrement phi E}} \text{ and } \eqref{eq:behavior phi},\\
        &\geq -M C - c_\varphi \qquad\text{since}\quad \int_0^\infty u(t,x)dx \leq 1. \\
    \end{align*}
    Therefore, taking $t_C = MC+1+c_\varphi$, we get, for any $t\geq t_C$,
    $$
    \int_0^t N_u(\tau)\,d\tau \geq \frac{1}{c_\varphi}>0.
    $$
\end{proof}
Using this lower bound on the source term, we construct an explicit subsolution which allows us to establish the Doeblin condition.
\begin{prop}\label{prop:minoration S(t)u}
		For any $ C\geq 0$, there exists a time $t'_C\geq 0$,  and a probability density $\mu\in L^1(\R)$ such that for any probability density $u_0\in E$  that satisfies  $\|u_0\|_E\leq C$, we have
		\[
		\fa t \geq t_C',\quad S(t)u_0\geq \eta(t) \times \mu.
		\]
		where $\eta$ is a strictly positive continuous function. 
	\end{prop}

\begin{proof}
	    We begin with a change of variables to remove the drift term from \eqref{eq:FPSL lin}.
    Next, we prove that the solution dominates a suitable subsolution of the heat equation on $[-1,1]$, whose source term satisfies the lower bound in \eqref{eq: minoration int N_u}. Diffusion effects then provide the probability density $\mu$ and the positive constant $\eta>0$.\\
	    
		Let $C>0$ and $u_0$ be a data satisfying the given hypothesis. According to Propostion \ref{prop:Minoration source}, there exists a time $t_C\geq 0$ such that for any $t\geq t_C$ the integrated source term satisfies $\int_0^t N_u(\tau)\,d\tau \geq \frac{1}{c_\varphi}$. \\
		
		Let us consider the function $W(t,x) =\exp\big(-\frac{1}{2}\mathcal{H}(x)\big)u(t,x)$, it satisfies
		$$
		\fa x\in\R, t\geq 0, \quad \partial_tW(t,x) = -\frac{1}{2}\big(h'(x)+\frac{1}{2}h(x)^2\big)W(t,x) +\partial_x^2W(t,x) +N_u(t)\delta_0(x).
		$$
		Now, let $M = \frac{1}{2}\sup_{x\in[-1,1]}\big|h'(x)+\frac{1}{2}h(x)^2\big|$. Then, for any $t\geq 0$ and $x\in [-1,1]$, we have
		$$
		\partial_tW(t,x)\geq -MW(t,x)+\partial_x^2W(t,x) +N_u(t)\delta_0(x).
		$$
		We consider $\Tilde{W}(t,x) = W(t,x)e^{Mt}$, and $\Tilde{N}_u(t) = N_u(t)e^{Mt}$ then multiplying by $e^{Mt}$ the above equation, we obtain 
		$$
		\fa x \in [-1,1], t\geq 0, \qquad \partial_t\Tilde{W}(t,x)\geq \partial_x^2\Tilde{W}(t,x)+\Tilde{N}_u(t)\delta_0(x).
		$$
		
		Let $\Phi$ be the solution to the  heat equation with Dirichlet conditions  on $[-1,1]$ with a Delta function as data.
		\begin{equation}
			\begin{cases}
				\partial_t \Phi(t,x) - \partial_x^2 \Phi(t,x) = 0, & \text{for } t > 0,\ x \in (-1,1), \\
				\Phi(t,-1) = \Phi(t,1) = 0, & \text{for } t > 0, \\
				\Phi(0,x) = \delta_0(x), & \text{for } x \in [-1,1].
			\end{cases}
		\end{equation}
		Let $\underline{W}(t,x) = \int_0^t\tilde{N}_u(s)\Phi(t-s,x)ds$. Then it satisfies
		\begin{equation}
			\begin{cases}
				\partial_t \underline{W}(t,x) - \partial_x^2 \underline{W}(t,x) = \tilde{N}_u(t)\delta_0(x), & \text{for } t > 0,\ x \in (-1,1), \\
				\underline{W}(t,-1) = \underline{W}(t,1) = 0, & \text{for } t > 0, \\
				\underline{W}(0,x) = 0, & \text{for } x \in [-1,1].
			\end{cases}
		\end{equation}
		By the maximum principle, since $\Tilde{W}\geq \underline{W}=0$ on $\partial\big(\R^+\times [-1,1]\big)$,we have $\Tilde{W}\geq \underline{W}$ on $\R^+\times [-1,1]$.
		\\

		Since $\underline{W}$ satisfies the heat equation on $[-1,1]$ with Dirichlet boundary conditions, we can find a time $t_C'\geq 0$, $\alpha>0$ and $\delta>0$ such that for all $|x|\leq \frac{1}{2}$ and $t\geq t_C'$, we get
		$$
		\underline{W}(t,x)\geq \delta e^{-\alpha t}.
		$$
		Hence, for $t'_C=\max(4,2t_C)$ and $t\geq t_C'$,  $u(t)$ satisfies 
		$$\fa |x|\leq \frac{1}{2}, \quad u(t,x)\geq \frac{1}{2\sqrt{2}c_\varphi}\,e^{-(\alpha+M)t}\times e^{\frac{1}{2}\mathcal{H}(x)} .$$
		Now, writing $\mu(x) =\bigg(\int_{-1}^1e^{\frac{1}{2}\mathcal{H}(y)}dy\bigg)^{-1}e^{\frac{1}{2}\mathcal{H}(x)}\mathbf{1}_{[-1,1]}(x)$ and $\eta(t) = \frac{1}{2\sqrt{2}c_\varphi}\,e^{-(\alpha+M)t}>0$ we get
		$$
		\forall t\geq t_C', \qquad u(t)\geq \eta(t)\times \mu.
		$$
	\end{proof}

\subsection{Uniform Local Mass Bound}
We now prove that $\|u(t)\|_E$ is uniformly controlled by the initial condition $\|u_0\|_E$.

\begin{lemma}\label{lem:estimation u_E}
    There is a constants $K_1\geq 0$ such that for any initial probability density $u_0\in E$ , we have for every $t\geq 0$
    \begin{enumerate}
    \item $\|(-x)_+u(t)\|_{L^1}\leq \|(-x)_+u_0\|_{L^1}e^{-t}+K_1.$
    \item $\|u(t)\|_{E}\leq (K_1 +1)\|u_0\|_{E}.$
    \end{enumerate}
\end{lemma}
\begin{proof}
    Let $\chi \in \mathcal{C}^\infty(\R)$ such that $\chi\big|_{(-\infty,0]} = -x$ and $\chi\big|_{[1,\infty)} = 0$. Then, there exists $a_1>0$, $c_1>0$ such that $\chi$ satisfies for all $ x\in \R$
    \[
     (-x)_+ -a_1\leq \chi(x)\leq (-x)_+, \qquad h(x)\chi'(x)\leq -\chi(x)+c_1,
    \]
    thanks to \eqref{eq:comportement h -infty}. Let
\[
I(t) = \int_\R\chi(x)u(t,x)\,dx, \qquad J(t) = \int_\R(-x)_+u(t,x)\,dx.
\]Then, for all $t \geq 0$
    \[
    J(t)-a_1\leq I(t)\leq J(t).
    \]
    Next, differentiating the integral, using that $\chi(0)=0$, we get, for all $t\geq 0$
    \begin{align*}
        I'(t)
        &= \int_\R \chi(x)\big(-\partial_x\big(h(x)u(t,x)\big)+\partial_x^2u(t,x) + N_u(t) \delta_0 \big)\,dx 
        \\
        &\leq\int_\R\big(h(x)\chi'(x)+\chi''(x)\big)u(t,x)\,dx
        \\
        &\leq \int_\R \big( -\chi(x) + c_1 + \chi''(x) \big) u(t,x) \, dx \leq -I(t) + c_1 + \|\chi''\|_{\infty},
    \end{align*}
where we used mass conservation, $\int_\R u(t,x)\,dx = 1$. Gronwall's Lemma gives
    \begin{align*}
        \fa t\geq 0, \qquad I(t)&\leq I(0)e^{-t}+c_1+\|\chi''\|_{\infty},
    \end{align*}
    and thus, since $J(t) \leq I(t) + a_1$ and $I(0) \leq J(0)$, for all $ t \geq 0$ we have
    \begin{align*}
        J(t)\leq I(0)e^{-t}+c_1+\|\chi''\|_{\infty}+a_1 
        \leq J(0)e^{-t}+c_1+\|\chi''\|_{\infty}+a_1.
    \end{align*}
    By setting $K_1 = c_1+\|\chi''\|_{L^\infty}+a_1$, this establishes the first point of the lemma.

    For the second point, we write, using mass conservation, 
    \begin{align*}
        \fa t \geq 0, \quad \|u(t)\|_{E} &= \|(-x)_+u(t)\|_{1}+\|u(t)\|_{1} =J(t)+1
        \\
        &\leq J(0)e^{-t}+K_1+1 \leq \|(-x)_+u_0\|_{1}+K_1+1\\
        &\leq \|u_0\|_{E} + K_1\leq (K_1+1)\|u_0\|_{E}.
    \end{align*}
    This is  the second point of the lemma.
\end{proof}
\subsection{Global contraction of $S(T)$ and exponential convergence}
    
    We complete the proof of Theorem~\ref{th:conv DH}. The computations below are standard in the Doeblin--Harris method and are strongly inspired by the original article \cite{zbMATH07654553}, where the method was first introduced. We adapt it to our setting by using the estimates provided in Lemma~\ref{lem:estimation u_E}.
We first announce a direct consequence of Proposition \ref{prop:minoration S(t)u}.
\begin{corollary}\label{cor:contraction 1}
Let $C>0$, and let $f\in E$ be such that 
\[
\int_{\mathbb{R}} f(x)\,dx = 0
\qquad\text{and}\qquad 
\|(-x)_+ f\|_{1} \le \frac{C}{2}\,\|f\|_{1}.
\]
Then, with the notation of Proposition~\ref{prop:minoration S(t)u}, we have
\[
\forall\, t \ge t'_C, \qquad 
\|S(t) f\|_{1} \le (1 - \eta(t))\,\|f\|_{1}.
\]

\end{corollary}

\begin{proof}
Since $\int_{\mathbb{R}} f(x)dx = 0$, we have the decomposition
\[
f = \frac{\|f\|_1}{2}\,(\nu_+ - \nu_-) \qquad \text{with} \quad \nu_+ = \max\!\left(0,\,\frac{2f}{\|f\|_1}\right), \quad \nu_- = \max\!\left(0,\,-\frac{2f}{\|f\|_1}\right).
\]

The assumption $\|(-x)_+ f\|_{1} \le \frac{C}{2}\|f\|_1$ implies that 
\(
\|(-x)_+ \nu_\pm\|_1 \le C,
\)
and by Proposition~\ref{prop:minoration S(t)u}, there exists $t'_C \ge 0$ such that, 
\[
\forall t \ge t'_C, \qquad  S(t)\nu_\pm \ge \eta(t)\,\mu.
\]

For any $t \ge t'_C$, we compute
\begin{align*}
\|S(t)f\|_{1}
&= \left\| S(t)\left( \frac{\|f\|_1}{2}(\nu_+ - \nu_-)\right)\right\|_{1} = \frac{\|f\|_1}{2}\,\| S(t)\nu_+ - \eta(t)\mu - (S(t)\nu_- - \eta(t)\mu )\|_1 \\
&\le \frac{\|f\|_1}{2}
  \big( \|S(t)\nu_+ - \eta(t)\mu\|_1 + \|S(t)\nu_- - \eta(t)\mu\|_1 \big) \\
&\le \frac{\|f\|_1}{2}
  \left( \int_{\mathbb{R}} (S(t)\nu_+ - \eta(t)\mu)\,dx
        + \int_{\mathbb{R}} (S(t)\nu_- - \eta(t)\mu)\,dx \right)
\quad (\text{since } S(t)\nu_\pm - \eta(t)\mu \ge 0)
\\
&\le (1 - \eta(t))\,\|f\|_1.
\end{align*}
This proves the result.
\end{proof}

\begin{proof}\noindent {\em Proof of Theorem \ref{th:conv DH}}\\
    Let $\beta >0$ and for $f\in E$, we define $N_\beta(f) = \|f\|_1+\beta\|(-x)f\|_1 $. Then, $N_\beta$ is equivalent to the norm $\|\cdot\|_E$ and satisfies
    \[
    \fa f\in E, \quad \min(\beta,\beta^{-1})\|f\|_E\leq N_\beta(f)\leq \max(\beta,\beta^{-1})\|f\|_E.
    \]
    We aim to show that for a suitable $\beta>0$, we can find a time $T>0$ such that $f\mapsto S(T)f$ is a contraction for the norm $N_\beta$.\\
    
   Let $C>0$ and let $f\in E$ satisfying $\int_\R f(x)\,dx = 0$. \\
   
 \noindent{\em First case: $\|(-x)_+f\|_1\leq \frac{C}{2}\|f\|_1$.}\\
    Then, according to Corollary \ref{cor:contraction 1},
    \[
    \fa t \geq t_C', \quad \|S(t)f\|_1\leq (1-\eta(t))\|f\|_1.
    \]
    So, for $t\geq t_C'$, using Corollary \ref{cor:contraction 1} and the first point of Lemma \ref{lem:estimation u_E} applied to $u(t) = \frac{S(t)f}{\|f\|_1}$, we get
    \[
    \big\|(-x)_+\frac{S(t)f}{\|f\|_1}\big\|_1\leq \big\|(-x)_+\frac{f}{\|f\|_1}\big\|_1e^{-t}+K_1.
    \]
%    and so
%    \[ \big\|(-x)_+S(t)f\big\|_1\leq \big\|(-x)_+f\big\|_1e^{-t}+K_1\|f\|_1.\]
    Using this estimate, we can bound the norm $N_\beta$ of $S(t)f$ as
    \begin{align*}
        N_\beta(S(t)f)&=\|S(t)f\|_1+\beta\|(-x)_+S(t)f\|_1\\
        &\leq (1-\eta(t))\|f\|_1+\beta\big(\|(-x)_+f\|_1e^{-t}+K_1\|f\|_1\big)\\
        &\leq \big(1+\beta K_1 -\eta(t)\big)\bigg[\|f\|_1+\frac{\beta e^{-t}}{1+\beta K_1 - \eta(t)}\|(-x)_+f\|_1\bigg]\\
        &\leq \big(1+\beta K_1 -\eta(t)\big)\bigg[\|f\|_1+\frac{\beta e^{-t}}{1 - \eta(t)}\|(-x)_+f\|_1\bigg] .
    \end{align*}
    Recall that $\eta(t) = \frac{1}{2\sqrt{2}c_\varphi}\,e^{-(\frac{\pi^2}{4}+M)t}$, then $\lim_{t\rightarrow \infty}\frac{e^{-t}}{1-\eta(t)} = 0$ so we can choose $T_C\geq t_C'>0$ such that $\frac{e^{-T_C}}{1-\eta(T_C)}\leq 1$. Now, we take $\beta\in (0, \frac{\eta(T_C)}{2K_1})$ so we get
    \[
    N_\beta(S(T_C)f)\leq (1-\frac{1}{2}\eta(T_C))\bigg[\|f\|_1+\beta\|(-x)_+f\|_1\bigg] = \gamma^{(1)}_CN_\beta(f).
    \]
 \noindent{\em Second case: $\|(-x)_+f\|_1\geq \frac{C}{2}\|f\|_1$.}\\
    Let $t\geq 0$ and $\varepsilon>0$. Using the first point of Lemma \ref{lem:estimation u_E} with $u(t) = \frac{S(t)f}{\|f\|_1}$ to get
    \begin{align*}
        N_\beta(S(t)f) &= \|S(t)f\|_1+\beta\|(-x)S(t)f\|_1\\
        &\leq \|f\|_1+\beta\big(\|(-x)_+f\|_1e^{-t}+K_1\|f\|_1\big)\\
        &\leq (1-\varepsilon+\beta K_1)\|f\|_1+\varepsilon\|f\|_1+\beta e^{-t}\|(-x)_+f\|_1\\
        &\leq (1-\varepsilon+\beta K_1)\|f\|_1+\big(\frac{2\varepsilon}{C}+\beta e^{-t}\big)\|(-x)_+f\|_1\\
        &\leq \big(1-\varepsilon+\beta K_1\big)\bigg[\|f\|_1+\beta \big(\frac{2\varepsilon}{\beta C}+e^{-t}\big)\|(-x)_+f\|_1\bigg].
    \end{align*}
    To get a contraction again, we choose $\beta$ such that $\beta K_1<\frac{1}{4}$ and $\varepsilon = 2\beta K_1$. We write $\gamma^{(2)}_{K_1} = 1-\beta K_1$ and we obtain
    \[
    N_\beta(S(t)f)\leq \gamma^{(2)}_{K_1}\bigg[\|f\|_1+\beta \big(\frac{4K_1}{ C}+e^{-t}\big)\|(-x)_+f\|_1\bigg].
    \]
    Finally, the choice of time $T'=\log(2)$ and $C=8K_1$ ensures that $\frac{4K'}{C}+e^{-t}\leq 1$ and that 
    \[
    N_\beta(S(T')f)\leq \gamma^{(2)}_{K_1}N_\beta(f).
    \]
    We then take $\gamma = \max\{\gamma^{(1)}_{K_1},\gamma^{(2)}_{K_1}\}\in (0,1)$ and $T = \max\{T_C, T'\}$ and obtain
    \[
    \fa f\in E, \quad N_\beta(S(T)f)\leq \gamma N_\beta(f).
    \]
    By an immediate induction, we get
    \[
    \fa f\in E, \fa k \in \N, \quad N_\beta(S(kT)f)\leq \gamma^kN_\beta(f).
    \]
    For any $t\geq 0$, we can write $t= \big[\frac{t}{T}\big]T +r$ with $r\in [0,T)$ and we get, with $\lambda = -\frac{\log\gamma}{T}$, 
    \begin{align*}
        N_\beta(S(t)f)&=N_\beta(S(\big[\frac{t}{T}\big]T +r)f = N_\beta(S(\big[\frac{t}{T}\big]T)S(r)f)\\
        &\leq \gamma^{[\frac{t}{T}]}N_\beta(S(r)f)\\
        &\leq \gamma^{-1}e^{-\lambda t}N_\beta(S(r)f)\\
        &\leq \gamma^{-1}\max(\beta,\beta^{-1})e^{-\lambda t}\|S(r)f\|_E.
    \end{align*}
    Using the second point of Lemma \ref{lem:estimation u_E}, we get that
    \[
    \fa r \in [0,T), \quad \|S(r)f\|_E\leq K\|f\|_E.
    \]
    And then, we finally obtain
    \begin{align*}
        \|S(t)f\|_E &\leq \big(\min(\beta,\beta^{-1})\big)^{-1}N_\beta(S(t)f) \\
        & \leq K\gamma^{-1}\max(\beta^2,\beta^{-2})e^{-\lambda t}\|f\|_E\\
        & = Me^{-\lambda t}\|f\|_E.
    \end{align*}
    Let $f = u_0 - u_\infty$ where $u_0$ is a probability density and $u_\infty$ is the unique stationary state, then, for all $ t\geq 0$, $S(t)f = S(t)u_0 - u_\infty$ and 
    \[
    \|u(t)-u_\infty\|_{E}\leq Me^{-\lambda t}\|u_0-u_\infty\|_{E}.
    \]
    This completes the proof of Theorem \ref{th:conv DH}.
\end{proof}

\section{Conclusion}
\label{sec:conclusion}

Our study of a conservative linear Fokker–Planck equation with a superlinear drift at infinity is motivated by the Integrate-and-Fire (I\&F) model in neuroscience \cite{BrHa, BL_2003, BretteG2005}. The novelty comes from the mass flux at $+\infty$ which generates the activity serving as a source term in the equation.  We first establish the  well-posedness of weak solutions in $L^1$, once the boundary condition at infinity has been defined in a suitable weak sense. 
Existence is proved passing to the limit in an approximate problem with a truncated drift. Uniqueness relies on a regularization argument in the spirit of the methods developed in \cite{DiPL89, Lebris_Lions_2008, Figalli2008}. The diffusion term allows us to prove regularizing effects is space and time, and consequently to establish the relative entropy property. The technical difficulty comes from the singularity of the pointwise re-injection of the outgoing flux at $x = +\infty$. The Doeblin–Harris method allows us to prove exponential convergence toward the stationary state in the $L^1$ norm with a linear weight at $-\infty$.
\\

Several questions are not treated in this paper, for instance: several other regularizing effects might be considered, more general initial conditions are possible for the relative entropy property, coupled systems are used in the biophysical literature as well as refractory states.
\\

In a forthcoming paper \cite{PRS3}, we study the same problem problem when the network activity is not only used for a pointwise source but also to define a nonlinearity on the drift as in the usual I\&F model
\begin{equation}
		\left\{\begin{aligned}
			& \frac{\partial p}{\partial t}(t,v) +\frac{\partial}{\partial v}((bN(t)-v)p(t,v)) - \frac{\partial ^2p}{\partial v}(t,v) = \delta_{V_R}(v)N(t), \qquad t\geq 0, v\leq V_F,\\
			& p(t,V_F) = 0, \\
			& N(t) = -\frac{\partial p}{\partial v}(t,V_F). 
		\end{aligned}\right.
	\end{equation}
	It appears that the classical tools used for the nonlinear I\&F equation, namely the reduction to a Stefan-type problem followed by a fixed-point argument, as in \cite{CGGS}, cannot be directly applied in this setting. The question of finite time blow-up (\cite{CCP}) is also deeply changed.
\\

\noindent {\bf Acknowledgment.} CR and DS are supported by the Fondation Simone et Cino Del Duca.

\bibliographystyle{siam}  
\bibliography{BibDPSZ}

\end{document}